\documentclass[letterpaper, 11pt,  reqno]{amsart}
\pdfoutput=1

\usepackage[margin=1.2in,marginparwidth=1.5cm, marginparsep=0.5cm]{geometry}

\usepackage{amsmath,amssymb,amsthm,amsxtra, esint}

\usepackage{color}
\usepackage[implicit=true]{hyperref}
\usepackage{amssymb}
\usepackage{mathrsfs}

\usepackage{xsavebox}

\usepackage{upgreek}

\usepackage{cases}

\allowdisplaybreaks[2]

\usepackage{tikz}

\definecolor{gr}{rgb}   {0.,   0.69,   0.23 }
\definecolor{bl}{rgb}   {0.,   0.5,   1. }
\definecolor{mg}{rgb}   {0.85,  0.,    0.85}
\definecolor{yl}{rgb}   {0.8,  0.7,   0.}
\definecolor{or}{rgb}  {0.7,0.2,0.2}

\newtheorem{theorem}{Theorem} [section]

\newtheorem{lemma}[theorem]{Lemma}
\newtheorem{proposition}[theorem]{Proposition}
\newtheorem{remark}[theorem]{Remark}

\newtheorem{definition}[theorem]{Definition}
\newtheorem{corollary}[theorem]{Corollary}

\newcommand{\noi}{\noindent}
\newcommand{\Z}{\mathbb{Z}}
\newcommand{\R}{\mathbb{R}}
\newcommand{\T}{\mathbb{T}}
\newcommand{\bul}{\bullet}

\let\P= \undefined
\newcommand{\P}{\mathbf{P}}

\newcommand{\E}{\mathbb{E}}
\newcommand{\PP}{\mathbb{P}}

\newcommand{\Law}{\textup{Law}}

\newcommand{\F}{\mathcal{F}}

\newcommand{\al}{\alpha}
\newcommand{\be}{\beta}
\newcommand{\dl}{\delta}

\newcommand{\nb}{\nabla}

\newcommand{\Dl}{\Delta}
\newcommand{\eps}{\varepsilon}

\newcommand{\g}{\gamma}

\newcommand{\ld}{\lambda}

\newcommand{\s}{\sigma}
\newcommand{\Si}{\Sigma}
\newcommand{\ft}{\widehat}

\newcommand{\wt}{\widetilde}
\newcommand{\cj}{\overline}
\newcommand{\dx}{\partial_x}
\newcommand{\dt}{\partial_t}

\newcommand{\LRA}{\Longrightarrow}

\newcommand{\ta}{\theta}

\renewcommand{\l}{\ell}
\renewcommand{\o}{\omega}
\renewcommand{\O}{\Omega}

\newcommand{\les}{\lesssim}
\newcommand{\ges}{\gtrsim}

\newcommand{\jb}[1]
{\langle #1 \rangle}

\newcommand{\ind}{\mathbf 1}

\renewcommand{\S}{\mathcal{S}}

\newcommand{\gf}{\mathfrak{g}}

\newcommand{\M}{\mathcal{M}}

\newcommand{\N}{\mathbb{N}}
\newcommand{\NN}{\mathcal{N}}

\renewcommand{\H}{\mathcal{H}}

\newtheorem*{ackno}{Acknowledgements}

\newcommand{\ze}{\zeta}
\newcommand{\upze}{\upzeta}
\newcommand{\Zf}{\mathfrak{Z}}
\newcommand{\zf}{\mathfrak{z}}

\newcommand{\fs}{\mathsf{f}}
\newcommand{\fns}{\fs_{N, \psi}}

\newcommand{\Zs}{\mathsf{Z}}
\newcommand{\zs}{\mathsf{z}}
\newcommand{\us}{\mathsf{u}}

\newcommand{\Cs}{\mathsf{C}}

\newcommand{\Ys}{\mathsf{Y}}

\newcommand{\too}{\longrightarrow}

\newcommand{\I}{\mathcal{I}}
\newcommand{\J}{\mathcal{J}}
\newcommand{\If}{\mathfrak{I}}

\newcommand{\D}{\mathcal{D}}

\newcommand{\C}{\mathcal{C}}

\newcommand{\vp}{\varphi}
\newcommand{\bP}{\mathbb{P}}

\newcommand{\uu}{\mathbf{u}}

\newcommand{\fn}{f_{N, \psi}}

\newcommand{\sF}{\mathscr{F}}

\newcommand{\vv}{\mathbf{v}}

\numberwithin{equation}{section}
\numberwithin{theorem}{section}

\makeatletter
\@namedef{subjclassname@2020}{%
	\textup{2020} Mathematics Subject Classification}
\makeatother

\begin{document}
\baselineskip = 14pt

\title[Probabilistic well-posedness  beyond variance blowup I]
{Probabilistic well-posedness of dispersive PDEs  beyond variance blowup I:
Benjamin-Bona-Mahony equation}

\author[G.~Li, J.~Li, T.~Oh, and N.~Tzvetkov]
{Guopeng Li, Jiawei Li, Tadahiro Oh, and Nikolay Tzvetkov}

\address{
Guopeng Li, 
School of Mathematics and Statistics, Beijing Institute of Technology, Beijing 100081, China}

\email{guopeng.li@bit.edu.cn}

\address{
Jiawei Li, 
School of Mathematics\\
The University of Edinburgh\\
and The Maxwell Institute for the Mathematical Sciences\\
James Clerk Maxwell Building\\
The King's Buildings\\
Peter Guthrie Tait Road\\
Edinburgh\\ 
EH9 3FD\\
 United Kingdom}

\email{jiawei.li@ed.ac.uk}

%

\address{
Tadahiro Oh, School of Mathematics\\
The University of Edinburgh\\
and The Maxwell Institute for the Mathematical Sciences\\
James Clerk Maxwell Building\\
The King's Buildings\\
Peter Guthrie Tait Road\\
Edinburgh\\ 
EH9 3FD\\
 United Kingdom, 
 and 
 School of Mathematics and Statistics, Beijing Institute of Technology,
Beijing 100081, China
}

\email{hiro.oh@ed.ac.uk}

\address{
Nikolay Tzvetkov,
Ecole Normale Sup{\'e}rieure de Lyon\\
UMPA\\
UMR CNRS-ENSL 5669\\
46, all{\'e}e d'Italie\\ 
69364-Lyon Cedex 07\\ 
France}

\email{tzvetkov@ens-lyon.fr}



\begin{abstract}

We investigate a possible extension of probabilistic well-posedness theory
of nonlinear dispersive PDEs with random initial data beyond variance blowup.
As a model equation, we study 
the Benjamin-Bona-Mahony equation (BBM)
with Gaussian random initial data.
By introducing a suitable
vanishing multiplicative renormalization constant
on the initial data, 
we show that 
solutions to  BBM with the renormalized  Gaussian random initial data
beyond variance blowup
converge in law to 
a solution to the {\it stochastic} BBM forced by 
the derivative of
a spatial white noise.
By considering alternative renormalization, 
we show that 
solutions to the renormalized BBM with the frequency-truncated
Gaussian initial data
converges in law to 
a solution to the linear stochastic BBM
with the full Gaussian initial data, forced
by 
the derivative of
a spatial white noise.
This latter result holds
for 
 the Gaussian random initial data  of arbitrarily low regularity.
We also establish analogous results
for the stochastic BBM forced by a fractional derivative of 
a space-time white noise.

\end{abstract}


\maketitle

\tableofcontents

\section{Introduction}

\subsection{Overview}
In seminal works \cite{BO94, BO96}, 
in  constructing invariant Gibbs measures
for nonlinear Schr\"odinger equations (NLS), 
Bourgain 
initiated the probabilistic well-posedness study
of nonlinear dispersive PDEs with random initial data.
Over the last fifteen years, 
probabilistic well-posedness
of dispersive equations, broadly interpreted
with random initial data and\,/\,or (additive) stochastic forcing, 
has attracted much attention and has been studied intensively;
see, for example, 
\cite{BT1, CO, BT3, BOP2, GKO,  OTh2, 
 STz1,  Bring0, STz2,  OOT1, Bring1, 
OOT2, 
STzX, BCST}.
See also surveys \cite{BOP4, Tzv1, DNY5}.
In particular, over the last five years, 
we have witnessed several breakthrough
results
\cite{GKO2, DNY2, DNY3, 
BDNY}, 
introducing novel tools and ideas such as paracontrolled
calculus in the dispersive setting, 
random averaging operators, 
and the theory of random tensors.

In \cite{DNY2}, Deng, Nahmod, and Yue introduced the notion of 
probabilistic scaling critical regularity, 
which roughly\footnote{\label{FT:1}In order to make things a priori computable, 
a certain simplification such as a dyadic frequency restriction
was introduced in \cite{DNY2} in computing 
a probabilistic scaling critical regularity.}
 corresponds to the regularity
above which 
the second Picard  iterate\footnote{Strictly speaking, the second Picard iterate minus the linear solution.}
 (see, for example,  \eqref{exp2})
is smoother than the random linear solution
(or the stochastic convolution in the stochastically forced case);
see also \cite[Subsection 1.2]{FOW}.
This notion of probabilistic scaling criticality
has provided useful heuristics
for threshold regularities, regarding probabilistic local well-posedness
of some dispersive equations.
For example,  in~\cite{DNY3}, 
the authors introduced the theory 
of random tensors
and showed
that NLS with the gauge-invariant nonlinearity on $\T^d$:
\begin{align*}
i \dt u - \Dl u + |u|^{2k} u = 0, \quad k \in \N, 
\end{align*}

\noi
is almost surely locally well-posed
with respect to the Gaussian random initial data of the form\footnote{With obvious modifications
for the complex-valued setting on $\T^d$.}
\eqref{data1}
in the full probabilistically scaling subcritical regime.

As emphasized in \cite{DNY5}, 
 the probabilistic scaling heuristics 
  provides only a guiding principle
and should {\it not}
be understood that the actual threshold between 
almost sure local well-posedness and ill-posedness
is always given by 
a probabilistic scaling critical regularity.
Indeed, 
there are recent results
\cite{Forlano, OO, Liu}
on {\it variance blowups} 
for dispersive PDEs
(see \eqref{div1} and Remark~\ref{REM:var1}), 
showing that, due to  divergence 
of the variance of a certain (renormalized) stochastic term, 
the standard probabilistic local well-posedness theory, 
based on the first order expansion 
\cite{McK, BO96, DPD}
(see \eqref{exp1}) or its higher order variant
\cite{BOP3, OPTz, GKO2},  
breaks down 
before reaching the threshold regularities
predicted by the probabilistic scaling heuristics.
In these works, 
the discrepancy between 
the probabilistic scaling prediction and 
actual probabilistic well-\,/\,ill-posedness
comes from the simplification introduced in computing 
a probabilistic scaling critical regularity
 mentioned in Footnote \ref{FT:1};
see 
\cite[Remark 1.8]{OO}
and  \cite[discussion after Proposition 1.5]{Liu}.

In this paper, we aim to 
seek for  a possible 
extension
of probabilistic well-posedness theory
of dispersive PDEs beyond variance blowup
by working on  a concrete example
(namely the 
Benjamin-Bona-Mahony equation \eqref{BBM0}).
See Theorems \ref{THM:3}
and \ref{THM:3a}
for the statements of our main results.

\begin{remark}\rm

Related to the probabilistic well-posedness study, 
there 
are results 
on 
pathological behavior of solutions with low regularity random initial data
(such as almost sure norm inflation); see
\cite{STz0, OOT, OOPTz}
for further details.

\end{remark}

\subsection{Benjamin-Bona-Mahony equation}

As a model  equation, 
we consider the following Benjamin-Bona-Mahony equation (BBM)
 on  the circle $\T=\R / (2\pi \Z)$:
\begin{align}
\begin{cases}
\dt u-\partial_{txx}u+\partial_{x}u
+\partial_{x}(u^{2})  = 0  \\
u|_{t = 0} = u_0, 
\end{cases}
\quad  (t,x) \in \R\times \T,
\label{BBM0}
\end{align}

\noi
where $u$ is a real-valued unknown.
With 
$D=-i\partial_{x}$, define the operator 
$\vp(D)$ by 
\begin{align}
\vp(D)=
- (1-\partial_{x}^{2})^{-1}\partial_{x}.
\label{phi1}
 \end{align}

\noi
Namely, 
 $\vp(D)$ is the Fourier multiplier operator 
with symbol:
\begin{align}
 \vp(n)=\frac{-i n}{1+n^{2}}.
\label{phi2}
\end{align}

\noi
Then, we can rewrite \eqref{BBM0}
 as
\begin{align}
\begin{split}
\dt u 
& = \varphi(D)  u+ \NN(u) \\
:\! & = \varphi(D)  u+\varphi(D) (u^2). 
\end{split}
\label{BBM1}
\end{align}

\noi
In the following, we  use~\eqref{BBM0} and \eqref{BBM1}
interchangeably.

The BBM equation \eqref{BBM0},
also known as the regularized long-wave equation, 
is a model for the unidirectional propagation of long-crested surface water waves; see, for example, \cite{Pere, BBM, BB73}. 
In particular, in \cite{BBM}, it was proposed as an alternative model to the Korteweg-de Vries  equation. 
By exploiting the 
smoothing property of the operator $\vp(D)$ appearing in~\eqref{BBM1}, 
Bona and the fourth author 
\cite{BT09}
proved that~\eqref{BBM0}  is globally well-posed
in $H^s(\R)$, $s \ge 0$,  
whose argument also applies 
 to the 
 periodic setting, 
 yielding global well-posedness in $H^s(\T)$, $s \ge 0$;
see \cite{RO10} for the details.
See also \cite{BCS1,BCS2}.
Furthermore, the aforementioned well-posedness results
are sharp in the sense that \eqref{BBM0}
is ill-posed in $H^s(\M)$ for $s < 0$, $\M = \R$ or $\T$; 
see 
 \cite{BT09, 
 PA11,  BD16, Forlano}.
In the following, we restrict our attention to the periodic setting.

In an effort to extend the well-posedness theory to Sobolev spaces
of negative regularity, Forlano \cite{Forlano}
studied the well-posedness issue of \eqref{BBM1} with  the following Gaussian\footnote{By imposing
appropriate conditions, 
the Gaussianity assumption on $g_n$
may be dropped to obtain the results in \cite{Forlano}.
See \cite[Appendix A]{Forlano}.} random initial data $u_0 = u_0^\o(\al)$ of the form:\footnote{By convention, we endow
$\T$ with the normalized Lebesgue measure $ dx_{_\T} =  (2\pi)^{-1}dx$
such that we do not need to carry factors involving $2\pi$.}
\begin{align}
u_0 = u_0^\o (\al) = \sum_{n\in \Z} \frac{g_n(\o) }{\jb{n}^\al} e_n ,
\label{data1}
\end{align}

\noi
where $\al \in \R$, 
$\jb{n} = \sqrt{1+|n|^2}$, 
$e_n(x) = e^{inx}$, 
and 
$\{g_n\}_{n\in \Z}$ is a sequence of independent 
complex-valued standard Gaussian random variables on some probability space 
$(\Omega, \sF, \bP)$ conditioned  that 
$g_{-n} = \overline{g_n}$, $n \in \Z$.
It is easy to see that $u_0$ in \eqref{data1}
belongs almost surely  to~$W^{s, p}(\T) \setminus W^{\al - \frac 12, p}(\T)$ for any 
$s < \al - \frac 12$ and  $1\leq p \leq \infty$.
Note that for $\al \le \frac 12$, 
the random function $u_0$ in \eqref{data1}
does not belong to $L^2(\T)$, almost surely.

Let $\zf$ be the random linear solution to \eqref{BBM1}
with  the random initial data 
 $u_0 = u_0^\o(\al)$ 
in~\eqref{data1}:
\begin{align}
\zf (t)= S(t) u_0
= e^{t\varphi(D)}u_0, 
\label{lin0}
\end{align}

\noi
where 
 $S(t)=e^{t\varphi(D)}$ denotes the linear  propagator
for \eqref{BBM1}.
By  considering the first order expansion
\cite{McK, BO96, DPD}:
\begin{align}
u = \zf + v
\label{exp1}
\end{align}

\noi
and studying the equation satisfied by the remainder term $v = u - \zf$:
\begin{align}
\begin{cases}
\dt v  = \varphi(D)  v+\NN (v+\zf)\\ 
v|_{t = 0} = 0, 
\end{cases}
\label{BBM1a}
\end{align}

\noi
Forlano 
 \cite{Forlano}
showed that, when $\frac 14 < \al \le \frac 12$, 
 \eqref{BBM1} 
(hence \eqref{BBM0})
is almost surely locally well-posed
with respect to the random initial data $u_0$ in \eqref{data1}.
See also \cite[Theorem~1.2]{Forlano}\footnote{Theorem 1.5 in the arXiv version.}
for almost sure global well-posedness
of~\eqref{BBM1} when $\al = \frac 12$
(namely, slightly outside $L^2(\T)$), 
following the globalization argument developed in~\cite{GKOT}
for 
singular stochastic wave equations.

In the same paper \cite{Forlano}, Forlano also showed that 
the aforementioned almost sure local well-posedness result is sharp
in the  sense that {\it variance blowup} occurs for $\al \le \frac 14$; see~\eqref{div1} 
and Remark \ref{REM:var1} below.
In proceeding with the first order expansion \eqref{exp1}
or its higher order variant, 
it is crucial that the second Picard iterate $\Zf$, defined by\footnote{Once again, strictly speaking, 
$\Zf$ is the second Picard  iterate minus the linear solution $\zf$.
For simplicity, however, we refer to $\Zf$ as the second Picard  iterate.
The same convention applies to the rest of the paper.}
\begin{align}
\begin{split}
\Zf (t)
&  = \I \big(  \NN( \zf ) \big)(t)
  = \I \big(  \vp(D) (\zf^2) \big)(t)\\
:\!&  =  \int^t_0 S(t-t') \vp(D) (\zf^2)(t') dt',
\end{split}
\label{exp2}
\end{align}
 
 \noi
makes sense, at least,  as a $\D'(\T)$-valued continuous  (in time) function 
(if we aim to construct a solution $u$ to \eqref{BBM1} in $C([0, T]; H^s(\T))$
for some $T>0$ and $s \in \R$).
Here, 
$\NN(u)$ and $\zf$ are as in \eqref{BBM1}
and \eqref{lin0}, respectively, 
and 
$\I$ denotes the Duhamel integral operator for \eqref{BBM1}, 
In~\cite{Forlano}, 
the author   showed that, 
when $\al \le \frac 14$, 
the second Picard  iterate $\Zf (t)$ in \eqref{exp2}
does {\it not} exist as a spatial distribution (for 
any fixed\footnote{In the remaining part of this paper, 
we restrict our attention to positive times.} $t > 0$).
More precisely, 
given $N \in \N$, 
let $\Zf_N$ be the frequency truncated version of $\Zf$
defined by 
\begin{align}
\Zf_N (t)
 = \I \big(  \NN( \P_N \zf ) \big)(t), 
\label{exp3}
\end{align}

\noi
where
$\P_N$ 
is  the Dirichlet projector onto the frequencies 
$\{|n| \leq N\}$; see \eqref{Diri1}.
Then, it was shown in \cite{Forlano}
that, when $\al \le \frac 14$, 
we have 
\begin{align}
\lim_{N \to \infty} 
\E\Big[|\jb{  \Zf_N(t), \psi}|^2\Big] = \infty
\label{div1}
\end{align}

\noi
for any $t > 0$
and any non-zero test function $\psi \in \D(\T) = C^\infty(\T)$, 
where 
$\jb{f, g}$ 
denotes the $\D'(\T)$-$\D(\T)$ duality pairing.
See 
\cite[Remark 3]{Forlano}\footnote{Remark 1.4 in the arXiv version.}
and \cite[Subsection 3.2]{Forlano}.

\begin{remark}\label{REM:var1}\rm

We use the term {\it variance blowup}
to describe the situation, 
where (i)~the variance of a certain stochastic term diverges (see \eqref{div1}
for the BBM case) at a threshold
{\it but} (ii)~the analytical framework still continues to hold beyond the threshold.
See the introduction in \cite{Hairer}
for a further discussion.

Let us consider the BBM case.
It is easy to show that for $\frac 14 < \al \le \frac 12$, 
$\Zf_N$ in \eqref{exp3}  converges almost surely to $\Zf$ in \eqref{exp2} 
in $C([0, T]; H^{2\al-\eps}(\T))$ for any $\eps > 0$;\footnote{In the remaining part
of the paper, we use $\eps > 0$ to denote an (arbitrarily) small constant.}
see \cite{Forlano}.
For now, let us  ignore the divergence~\eqref{div1}
and {\it pretend} that the second Picard iterate 
$\Zf  =  \I \big(  \NN( \zf ) \big)$
exists
as an element in $C([0, T]; H^{2\al-\eps}(\T))$
even for $\al \le \frac 14$.
Then, 
by applying  the product estimate (Lemma \ref{LEM:BOZ})
to the following Duhamel formulation of \eqref{BBM1a}:
\begin{align*}
v(t) = \I \big(\NN(v)\big)(t)
+2 \I \big(\vp(D)(v \zf)\big)(t)
+ \Zf(t), 
\end{align*}

\noi
we could show that  \eqref{BBM1a}
would be  almost surely locally well-posed, at least for $\al > \frac 16$.\footnote{The condition
$\al > \frac 16$ appears in making sense of the product
$v \zf$ for $v \in C([0, T]; H^{2\al-\eps}(\T))$ 
and $\zf \in C([0, T]; W^{\al- \frac 12 -\eps, \infty}(\T))$, using Lemma \ref{LEM:BOZ}.} 
Namely,  
the analytical framework for proving almost sure
local well-posedness of \eqref{BBM1a}
extends beyond the threshold $\al = \frac 14$.

In recent years, 
variance blowup phenomena have been 
 observed
in various contexts;
SDEs~\cite{CQ}, 
 the fractional KPZ equation 
\cite[Subsection 4.9]{Hoshino}, 
the stochastic wave and heat equations~\cite{OO}
(see also \cite{Deya2} for a related result), 
and NLS \cite{Liu}.
We note that 
this phenomenon has been observed
only in the situation
when
(i)~given random initial data is rougher than the Gaussian free field 
or (ii)~given stochastic forcing is rougher than that of a space-time white noise.
In the case of nonlinear dispersive PDEs with random initial data
or additive stochastic forcing, 
variance blowup in particular implies that 
 the standard probabilistic well-posedness theory, 
based on the first order expansion (as in~\eqref{exp1})
or its higher order variant (see, for example, \cite{BOP3, OPTz,  GKO2}), 
completely breaks down. 
See~\cite{OO} for a further discussion.

\end{remark}

As mentioned above, 
our main goal in this paper is 
to investigate a possible 
extension
of probabilistic well-posedness theory
of BBM \eqref{BBM0}
beyond the variance blowup threshold $\al = \frac 14$.
By introducing a suitable
vanishing multiplicative renormalization constant
on the  initial data (see \eqref{BBM2}), 
we will show that 
solutions to BBM with Gaussian random initial data
beyond variance blowup
converge in law to 
a solution to the {\it stochastic} BBM forced by 
the derivative of
a 
spatial white noise;
see Theorem \ref{THM:3}.
See also Theorem 
\ref{THM:3a} where we 
consider   alternative renormalization
and study the limiting behavior of  the weakly interacting BBM \eqref{BBM8}.

\begin{remark}\rm

We point out that BBM \eqref{BBM0}
is {\it self-renormalizing}
due to the quadratic nonlinearity with a derivative, 
just as in the case of the (stochastic) KdV equation \cite{KPV, BO97, Oh1, Oh2, Oh3, OQS}
and the stochastic Burgers equation \cite{GP}.

Let us consider \eqref{BBM1}.
Fix 
$\frac 14 < \al \le \frac 12$.
Then, from \eqref{lin0}
with 
 \eqref{data1} and the fact that $\vp(-n) = -\vp(n)$, we have 
\begin{align*}
\eta_N = \E\big[(\P_N \zf(x))^2\big]
= \sum_{|n|\le N} \frac{1 }{\jb{n}^{2\al}} 
\too \infty, 
\end{align*}

\noi
as $N \to \infty$.
This  shows that $(\P_N \zf)^2$ diverges as $N \to \infty$, 
forcing us 
 to consider the renormalized power
$(\P_N \zf)^2 - \eta_N$.
It follows from a standard computation with 
Lemma~\ref{LEM:OOT}
that 
$(\P_N \zf)^2 - \eta_N$ converges almost surely 
in $C([0, T]; W^{2\al - 1- \eps, \infty}(\T))$ as $N \to \infty$.
Then, from~\eqref{phi1} and \eqref{phi2} (in particular, noting that $\vp(n)|_{n = 0} = 0$), 
we see that 
\begin{align}
 \NN(\P_N \zf ) 
= \varphi(D) \big((\P_N \zf )^2 \big)
= \varphi(D) \big((\P_N \zf )^2  - \eta_N\big)
\label{FO1}
\end{align}

\noi
converges 
 almost surely 
in $C([0, T]; W^{2\al - \eps, \infty}(\T))$ as $N \to \infty$.
Namely, thanks to the presence of the operator $\vp(D)$
in $\NN(u) = \vp(D)(u^2)$, 
the renormalization constant $\eta_N$
disappears in~\eqref{FO1}
and, hence, 
$ \NN(\P_N \zf ) $ (and thus $\Zf_N = \I\big( \NN(\P_N \zf ) \big)$)
converges without (explicit) renormalization; see
\cite[Remark 2]{Forlano}\footnote{Remark 1.3 in the arXiv version.}
for a related discussion.

\end{remark}

\begin{remark}\rm
We also mention
the  related  works
 \cite{DS1, DS2, DST}
on BBM \eqref{BBM0} with random initial data.

\end{remark}

\subsection{Renormalization of BBM beyond variance blowup}

In the study of singular stochastic parabolic PDEs, 
 a  variance blowup phenomenon
has been observed for the following fractional KPZ equation on $\T$:\footnote{For simplicity
of the presentation, we ignore the renormalization ``$(\dx h)^2 - \infty$''
on the quadratic nonlinearity $(\dx h)^2$ in \eqref{KPZ1}.}
\begin{align}
\dt h = \dx^2 h  + (\dx h)^2 + |\dx|^{\al}\xi
\label{KPZ1}
\end{align}

\noi
for  $\al \ge \frac 14$, where $\xi$ denotes a space-time white noise
on $\R_+\times \T$;
see \cite[Subsection 4.9]{Hoshino}.
In a recent work~\cite{Hairer}, 
Hairer 
introduced 
a new   renormalization procedure beyond the variance blowup
(by considering the  $\al = 1$ case in \eqref{KPZ1}), 
which we explain below.
Given $N \in \N$, 
consider the following fractional 
KPZ (with $\al = 1$):
\begin{align}
\dt h_N = \dx^2 h_N + \big((\dx h_N)^2 - C_N \big) + N^{-\frac 34}\P_N  \dx \xi, 
\label{KPZ2}
\end{align}

\noi
where
the frequency-truncated
noise $\P_N  \dx \xi$
is endowed with 
a vanishing multiplicative renormalization constant
$ N^{-\frac 34}$.
Then, Hairer showed that solutions to \eqref{KPZ2}
converge {\it in law}
to a solution to the standard KPZ forced by a space-time white noise
(namely, \eqref{KPZ1} with $\al = 0$) as $N \to \infty$
(with an appropriate assumption on deterministic initial data).
A key ingredient for establishing this convergence result is the  observation
 that 
the second Picard iterate 
for the  fractional KPZ \eqref{KPZ2} 
with the renormalized noise $ N^{-\frac 34}\P_N  \dx \xi$
converges in law to the stochastic convolution
(= the linear solution)
for the standard KPZ~\eqref{KPZ1} with $\al = 0$:
\[ \int_0^t e^{(t - t')\dx^2} \xi (dt'), \]

\noi
as $N \to \infty$.
See \cite[Proposition 2.2]{Hairer}
for a precise statement (at the level of noises).
In \cite{Hairer}, 
this latter convergence result on the second Picard iterates 
was established by  a central limit theorem
via 
 the {\it fourth moment theorem}
due to  Nualart and Peccati
 \cite{NP05};   see Lemma \ref{LEM:FMT}.
In the same paper \cite{Hairer}, 
Hairer also introduced renormalization beyond variance blowup
in the SDE context;
see  \cite[Proposition~3.11]{Hairer}.

\begin{remark} \rm 
(i) 
In \cite{GT25}, Gerencs\'er and Toninelli
also studied the issue of renormalization beyond variance blowup
for the fractional KPZ \eqref{KPZ1}.
They considered  the following renormalized version:
\begin{align}
\dt h_N = \dx^2 h_N + N^{\frac 12 - 2\al} (\dx h_N)^2 - C_N  + \P_N  |\dx|^\al  \xi, 
\label{KPZ3}
\end{align}

\noi
where 
a vanishing multiplicative renormalization constant now appears on the nonlinearity.
By introducing an appropriate correction term 
(depending on $ \P_N  |\dx|^\al  \xi$)
on initial data, 
they showed that, for any $\al >  \frac 14$,   solutions 
to \eqref{KPZ3}
converges in law to a solution to the following linear equation:
\begin{align*}
\dt h = \dx^2 h + c_\al  \xi_1  +  |\dx|^\al  \xi_2, 
\end{align*}

\noi
where $\xi_1, \xi_2$ are independent copies of  a  space-time white noise.

\smallskip

\noi
(ii) 
In the study of stochastic parabolic PDEs, 
(super-)critical models,
such as
the KPZ equation and the stochastic heat equation (with a multiplicative space-time white noise)
in 
the higher dimensional setting, 
have attracted a wide attention in recent years.
Phenomena analogous to 
variance blowup have been observed
and 
renormalizations analogous to \eqref{KPZ2}
or \eqref{KPZ3} have been implemented;
see, for example, 
\cite{BC, MSZ, CSZ17, GRZx, CD, CCM, GQT, DG, CSZ2, GRZ, GT25}.

\end{remark}

\medskip

Let us now turn our attention to the BBM equation 
\eqref{BBM0}  with the random initial data $u_0 = u_0^\o(\al)$ in ~\eqref{data1}
beyond the variance blowup, following Hairer's approach \cite{Hairer}.
Fix $ \al \le \frac 14$.
Given $N \in \N$, we define 
the renormalization constant 
  $C_{\al, N}$ by setting 
\begin{align}
C_{\al, N}= 
\bigg( \sum_{|n|\le N }
\frac{2}{\jb{n}^{4\al}}\bigg)^{-\frac 14}
\sim 
\begin{cases}
(\log \jb{N})^{-\frac 14} & \text{for } \al = \frac 14,\\
\jb{N}^{\al-\frac 14}
  &  \text{for }   \al<\frac 14.
\end{cases}
\label{CN}
\end{align}

\noi
Note that 
\begin{align}
C_{\al, N} \too 0\quad \text{as}\quad N \too \infty.
\label{CN1}
\end{align}

\noi
We then  consider the  BBM equation 
\eqref{BBM0}
with the renormalized (frequency-truncated) random initial data:
\begin{align}
\begin{cases}
\dt u_N-\partial_{txx}u_N+\partial_{x}u_N
+\partial_{x}(u_N^{2})  = 0  \\
u_N|_{t=0} = C_{\al, N} \P_{N} u_0, 
\end{cases}
\label{BBM2}
\end{align}

\noi
or equivalently, 
\begin{align}
\begin{cases}
\dt u_N  = \vp(D) u_N + \NN(u_N)  \\
u_N|_{t=0} = C_{\al, N} \P_{N} u_0, 
\end{cases}
\label{BBM3}
\end{align}

\noi
where
 $u_0 = u_0^\o(\al)$ is the Gaussian random initial data  in \eqref{data1}.
Here,  $\vp(D)$ and $\NN(u)$ are  as in~\eqref{phi1} and \eqref{BBM1}, respectively.
We note that, thanks to the frequency truncation on the initial data,~\eqref{BBM2} (and~\eqref{BBM3}) is globally well-posed.
In the following, we study the limiting behavior of the solution $u_N$ to \eqref{BBM2} as $N \to \infty$.

Consider the following  second order expansion for $u_N$:
\begin{align}
u_N
&  =  z_N +Z_N + v_N, 
\label{exp3x}
\end{align}

\noi
where $z_N$
denotes the random linear solution given by 
\begin{align}
\begin{split}
z_N (t)  
& = S(t) C_{\al, N} \P_N u_0\\
& = \sum_{n \in \Z} e^{t\vp(n)} 
C_{\al, N} 
\frac{g_n}{\jb{n}^\al}
\ind_{|n|\le N}\cdot 
 e_n
 \end{split}
 \label{lin1}
\end{align}

\noi
and $Z_N$ denotes the second Picard iterate defined by 
\begin{align}
Z_N (t)
= \I\big(\NN(z_N) \big)(t)
 =  \int^t_0 S(t-t') \vp(D)  (z_N^2)(t') dt'.
 \label{I3}
\end{align}

\noi
Then, the remainder term 
$v_N = u_N - z_N - Z_N$ satisfies the following equation: 
\begin{align}
\begin{cases}
\dt v_N= 
\vp(D) v_N 
+ \vp(D) 
\big(
(v_N +  z_N +Z_N )^2 - z_N^2\big)\\
v_N|_{t=0} = 0.
\end{cases}
\label{BBM4}
\end{align}

From \eqref{lin1} with \eqref{lin0}, 
we have $z_N =C_{\al, N} \P_N \zf $.
In particular, it follows from  
 \eqref{CN1} that 
$z_N$ converges to $0$ as $N \to \infty$; see Lemma \ref{LEM:sto}.
Next, we study convergence of the second Picard iterate $Z_N$ in \eqref{I3}.
Define $Z$ to be the solution to the following linear 
stochastic equation, forced by 
 the derivative of a spatial white noise:
\begin{align}
\begin{cases}
\dt Z-\partial_{txx}Z+\partial_{x}Z
 = \dx \ze  \\
Z|_{t=0} =0, 
\end{cases}
\label{BBM4a}
\end{align}

\noi
or equivalently, 
\begin{align}
\begin{cases}
\dt Z = \vp(D) Z - \vp(D) \ze  \\
Z|_{t=0} = 0. 
\end{cases}
\label{BBM4b}
\end{align}

\noi
Here, 
$\ze$ denotes a spatial white noise on $\T$:
\begin{align}
\ze = \sum_{n\in \Z}{\gf_n}(\o) e_n, 
\label{I3b}
\end{align}

\noi
where 
 $\{\gf_n\}_{n\in \Z}$ is a sequence of independent 
 complex-valued standard Gaussian random variables conditioned that 
$\gf_{-n} = \cj{\gf_n}$, $n \in \Z$.
By writing \eqref{BBM4b} in the Duhamel formulation, 
we have 
\begin{align}
Z(t) =
-  \int^t_0 S(t-t') \vp(D) \ze dt'.
\label{lin2}
\end{align}

\noi
The following proposition
establishes 
a key convergence result of 
the second Picard iterate $Z_N$ in \eqref{I3}
to $Z$ in \eqref{lin2}.
We emphasize that this convergence takes place only {\it in law}
as in Hairer's work
\cite{Hairer}.

\begin{theorem}
\label{THM:1}

Let $\al\le \frac 14$.
Given any $s < \frac 12$, 
the second Picard iterate $Z_N$ defined in~\eqref{I3}
converges in law to the Gaussian process $Z$ defined in \eqref{lin2}
in $C(\R_+; W^{s,\infty}(\T))$
as $N \to \infty$, 
where
$C(\R_+; W^{s,\infty}(\T))$
is 
endowed with the
compact-open topology in time.

\end{theorem}

As in \cite[Proposition 2.2]{Hairer}, 
a proof of Theorem \ref{THM:1}
is based on the 
 fourth moment theorem
(Lemma \ref{LEM:FMT}). 
More precisely, 
we first use the fourth moment theorem
to prove uniqueness of  the limit of 
the $\D'(\T)$-valued stochastic process $Z_N$
by identifying the limit of 
a finite-dimensional marginals
$\{ Z_N(t_j)\}_{j = 1}^m$;
see Subsection \ref{SUBSEC:Z1}.
We then establish 
tightness of the sequence $\{Z_N\}_{N \in \N}$
(Proposition \ref{PROP:tight})
to conclude Theorem~\ref{THM:1}.
See Section \ref{SEC:Z} for details.

By formally taking the limit
$(z_N, Z_N) \to (0, Z)$ in \eqref{BBM4}, 
we obtain the following limiting equation:
\begin{align}
\begin{cases}
\dt   v=  \vp(D) v +  \vp(D) 
\big( (v +  Z)^2 \big) \\
 v|_{t=0} =0.
 \end{cases}
\label{BBM4c}
\end{align}

\noi
We now rewrite  \eqref{BBM4c}
at the level of the original BBM formulation.
Namely, 
 by setting 
\begin{align*}
u = Z + v
\end{align*}

\noi
with \eqref{BBM4a}, we see that $u$ satisfies 
the following stochastic BBM equation forced by 
the derivative of
the spatial white noise
 $\ze$  in \eqref{I3b}:
\begin{align}
\begin{cases}
\dt u-\partial_{txx}u+\partial_{x}u
+\partial_{x}(u^{2})  = \dx \ze  \\
u|_{t = 0} = 0.
\end{cases}
\label{BBM4d}
\end{align}

\begin{proposition}\label{PROP:2}

The stochastic BBM equation
 \eqref{BBM4d} is almost surely globally well-posed.
More precisely, 
there exists a set $\Si \subset \O$
with $\PP(\Si) = 1$ such that,
given any $\o \in \O$, 
there exists a unique global-in-time
solution
$u = u^\o \in C(\R_+; H^{\frac 12- \eps}(\T))$
to \eqref{BBM4d}
of the form $u = Z + v$, 
where $Z
\in C(\R_+; W^{\frac 12 - \eps,\infty}(\T))$
 is as in \eqref{lin2}
 and Theorem \ref{THM:1}
and $v$ is the unique solution to~\eqref{BBM4c}
in the class
$C(\R_+; H^1(\T))$.

\end{proposition}

Proposition \ref{PROP:2} follows
from a straightforward adaptation of the argument
in \cite{BT09, Forlano}, 
using
 the almost sure regularity of $Z$
in Theorem \ref{THM:1}.
We present a proof of Proposition~\ref{PROP:2} in  Section \ref{SEC:limit}.

\medskip

Finally, we  state our main result.

\begin{theorem}
\label{THM:3}

Let $\al\leq \frac14$.
As $N \to \infty$, 
the solution $u_N$
to \eqref{BBM2}
with the renormalized initial data
$C_{\al, N} \P_{N} u_0$ 
converges in law
to the solution $u$ to \eqref{BBM4d}, 
constructed in Proposition~\ref{PROP:2}, 
in $C(\R_+; H^{- \frac 14 - \eps}(\T))$.

\end{theorem}

\medskip

Theorem \ref{THM:3} presents the first probabilistic 
construction of solutions
 beyond variance blowup for dispersive equations
with random initial data;
see also Theorem~\ref{THM:3a} below.
In particular, we prove that the effect of the random initial data
turns into a stochastic forcing in the limiting equation~\eqref{BBM4c}, 
which is an interesting new phenomenon to observe.
Moreover, our approach 
holds for {\it any} $\al \le \frac 14$, even beyond
the threshold $\al_\text{crit} = 0$
given by the probabilistic scaling heuristics.
We also establish an analogous result
for the stochastic BBM forced by a fractional derivative of a space-time white noise
(see \eqref{BBM6} below).
See Remark~\ref{REM:SBBM1}\,(ii)
and Appendix~\ref{SEC:A}.

Once we establish convergence of 
the second Picard iterate $Z_N$ in \eqref{I3} to $Z$ in \eqref{lin2}  (Theorem \ref{THM:1}), Theorem~\ref{THM:3}
follows without much difficulty.
We first apply
the Skorokhod representation theorem
(Lemma~\ref{LEM:Sk})
to upgrade the convergence in law
of $(z_N, Z_N)$ to $(0, Z)$ 
to almost sure convergence 
of the associated stochastic terms (Lemma~\ref{LEM:sto}).
The rest follows from a standard PDE argument.
See Section \ref{SEC:conv} for a proof.

\medskip

Theorem \ref{THM:3} is in the spirit of Hairer's formulation  \cite{Hairer}, 
where we placed the vanishing renormalization constant $C_{\al, N}$
on the random initial data in 
\eqref{BBM2}
(instead of placing it on the noise 
as in \eqref{KPZ2}).
Let us  now consider a slightly different formulation
in the spirit of the work \cite{GT25} (see \eqref{KPZ3}), 
where we  place a vanishing renormalization constant on the nonlinearity.
More precisely, 
for $\al \le \frac 14$, we consider the following 
{\it weakly interacting} BBM:
\begin{align}
\begin{cases}
\dt \uu_N-\partial_{txx}\uu_N+\partial_{x}\uu_N
+C_{\al, N}^2 \partial_{x}(\uu_N^{2})  = 0  \\
\uu_N|_{t=0} =  \P_{N} u_0, 
\end{cases}
\label{BBM8}
\end{align}

\noi
where $C_{\al, N}$ and $u_0 = u_0^\o(\al)$ are
as in \eqref{CN} and \eqref{data1}, respectively.
The weakly coupling constant $C_{\al, N}^2$
in \eqref{BBM8}
is chosen to balance
with 
the strong bilinear interaction (fast oscillation)
of the low regularity 
random linear solution 
so that we obtain a finite non-trivial limit
(namely, the forcing $\dx \ze$ in \eqref{BBM9}).
In fact, it is easy to see that the second Picard  iterate
for \eqref{BBM8}
is once again given by $Z_N$ in \eqref{I3}
which converges in law to~$Z$ in \eqref{BBM4a}
as seen in Theorem \ref{THM:1}.
Then, by considering the first order expansion 
and formally taking a limit as $N \to \infty$, 
we obtain the following limiting equation:
\begin{align}
\begin{cases}
\dt \uu-\partial_{txx}\uu+\partial_{x}\uu
  =\dx \ze  \\
\uu|_{t = 0} = u_0, 
\end{cases}
\label{BBM9}
\end{align}

\noi
where $\ze$ is the spatial white noise in \eqref{I3b},
assumed to be
independent of $u_0$ in \eqref{data1}.
We note that the limiting equation is now linear but there are now
two (independent) sources
of randomness.  Compare this with \eqref{BBM4d}, 
which is nonlinear but with one source of randomness.

\begin{theorem}\label{THM:3a}

Let $\al\leq \frac14$.
As $N \to \infty$, 
the solution $\uu_N$
to the renormalized BBM~\eqref{BBM8}
with the truncated random initial data
$\P_N u_0$, where $u_0$ is as in  \eqref{data1}, 
converges in law
to the solution $\uu$ to~\eqref{BBM9}
in $C(\R_+; H^{\al - \frac 12 - \eps}(\T))$.

\end{theorem}

We first note that the limiting equation \eqref{BBM9}
is linear and thus is globally well-posed.
Moreover, a slight modification of the proof 
 of Theorem \ref{THM:1}
 shows that $(\P_N u_0, Z_N)$
 converges in law
 to the jointly Gaussian process\footnote{Here, only the second argument $Z$
 is time dependent.}
 $(u_0, Z)$
 in
 $W^{\al - \frac 12 - \eps, \infty}(\T)
 \times C(\R_+; W^{\frac 12 - \eps, \infty}(\T))$,
which in particular implies
 independence of  $u_0$ and $\ze$;
see Remark~\ref{REM:indep}.
Then, 
Theorem \ref{THM:3a}
follows from 
the first order expansion and 
a straightforward modification 
of the proof of Theorem~\ref{THM:3},
and 
thus we omit details.
In Appendix~\ref{SEC:A}, 
we also establish 
an analogous result (Theorem~\ref{THM:6})
for the stochastic BBM forced by a fractional derivative of a space-time white noise.

\medskip

Renormalization is
a way to tame infinity arising in an equation
due to a rough random source
and it may often appear in the form of subtracting a diverging constant
or 
recentering by projection onto a homogeneous Wiener chaos of a fixed
degree.
Renormalization can also appear in the form
of a vanishing (or diverging) multiplicative constant.
See, for example,~\cite{ORW}, 
where renormalization is given by the multiplication
by a vanishing constant.
See also \cite{HS, ORSW1, ORSW2, Zine}.
The divergence coming from  the variance blowup
for BBM can not be tamed by subtraction or recentering, 
and we indeed need 
renormalization  by the multiplication
of a vanishing constant, 
either 
on the random initial data
as in \eqref{BBM2}
or 
on the nonlinearity
as in~\eqref{BBM8}.

As mentioned above, 
Theorem \ref{THM:3} on the convergence of \eqref{BBM2}
to \eqref{BBM4d}
is a direct analogue of Hairer's work \cite{Hairer}
in the random initial data context
whose proof is slightly more involved than that of Theorem \ref{THM:3a}.
However, 
from the viewpoint of probabilistic well-posedness
of dispersive PDEs with low regularity random initial data, it may be more natural
to consider 
the  weakly interacting BBM \eqref{BBM8}
and view Theorem \ref{THM:3a}
as an extension
of probabilistic well-posedness
of BBM beyond variance blowup, 
since the limiting equation \eqref{BBM9} still retains 
the original 
{\it rough} random initial data $u_0 = u_0^\o(\al)$
of regularity $\al - \frac 12 - \eps$,
while the renormalized random  initial data 
$C_{\al, N} \P_{N} u_0$
in \eqref{BBM2}
is almost surely bounded in $W^{-\frac 14 - \eps, \infty}(\T)$;
see Lemma~\ref{LEM:sto}.
In particular, Theorem \ref{THM:3a}
presents an extension
of probabilistic well-posedness of BBM
beyond variance blowup
for the Gaussian random initial data  of arbitrarily low regularity.

\medskip

We conclude this introduction by stating several remarks.

\begin{remark}\rm

In this paper, we use the Dirichlet frequency projector
$\P_N$
for regularization.
We can also use mollification for regularization,
in which case the renormalization constant $C_{\al, N}$ in~\eqref{CN}
needs to be modified accordingly, 
depending on a mollification kernel;
see \cite[Remark~1.14]{GKO2}.

\end{remark}

\begin{remark}\rm

We note that 
two renormalized formulations
\eqref{BBM2}
and \eqref{BBM8} are {\it not} equivalent.
Let $u_N$ be the solution to \eqref{BBM2}
with the renormalized initial data
$C_{\al, N} \P_{N} u_0$ as in Theorem \ref{THM:3}.
Then, let 
$\us_N = C_{\al, N}^{-1} u_N$
such that $\us_N|_{t = 0} = \P_N u_0$.
Then, $\us_N$
satisfies the following equation:
\begin{align}
\begin{cases}
\dt \us_N-\partial_{txx}\us_N+\partial_{x}\us_N
+C_{\al, N} \partial_{x}(\us_N^{2})  = 0  \\
\us_N|_{t=0} =  \P_{N} u_0.
\end{cases}
\label{BBM8a}
\end{align}

\noi
Note that we have $C_{\al, N}$ on the nonlinearity
in \eqref{BBM8a}
(instead of $C_{\al, N}^2$ in \eqref{BBM8}).
 In particular, the solution $\us_N$ to \eqref{BBM8a}
 does not converge to any limit.

\end{remark}

\begin{remark}\label{REM:IV}\rm
 Given a deterministic function $v_0 \in H^1(\T)$, 
consider
\begin{align}
\begin{cases}
\dt u_N-\partial_{txx}u_N+\partial_{x}u_N
+\partial_{x}(u_N^{2})  = 0  \\
u_N|_{t=0} = v_0 + C_{\al, N} \P_{N} u_0, 
\end{cases}
\label{BBMx2}
\end{align}

\noi
where $C_{\al, N}$ and $u_0 = u_0^\o(\al)$ are
as in \eqref{CN} and \eqref{data1}, respectively.
Then, a straightforward modification of 
the proof of Theorem \ref{THM:3}
shows that the solutions $u_N$ to \eqref{BBMx2}
converge in law to the solution $u$ to the 
following stochastic BBM
forced by 
the derivative of
a spatial white noise:
\begin{align*}
\begin{cases}
\dt u-\partial_{txx}u+\partial_{x}u
+\partial_{x}(u^{2})  = \dx \ze  \\
u|_{t = 0} = v_0, 
\end{cases}
\end{align*}

\noi
where 
 $\ze$ is as  in \eqref{I3b}.
A similar comment applies to 
the weakly interacting BBM~\eqref{BBM8}
considered in 
Theorem \ref{THM:3a}.

\end{remark}

\begin{remark}\label{REM:SBBM1}\rm

\noi
(i) 
Given $ \al \le \frac 14$, 
consider 
\begin{align}
\begin{cases}
\dt u_N-\partial_{txx}u_N+\partial_{x}u_N
+\partial_{x}(u_N^{2})  = \ze_{\al, N}  \\
u_N|_{t = 0} = 0, 
\label{BBM5}
\end{cases}
\end{align}

\noi
where $\ze_{\al, N} = - \dx (z_N^2)$
with $z_n$ as in \eqref{lin1}.
Then, 
a slight modification of the proof of Theorem \ref{THM:3} shows that 
solutions $u_N$ to 
\eqref{BBM5} converge in law to the solution $u$
to \eqref{BBM4d}.

\smallskip

\noi
(ii) 
Given $\al \in \R$, 
consider the following stochastic BBM:
\begin{align}
\dt u-\partial_{txx}u+\partial_{x}u
+\partial_{x}(u^{2})  = \jb{\dx}^{\al} \dx \xi,   
\label{BBM6}
\end{align}

\noi
where $\xi$ denotes the space-time white noise on $\R_+\times \T$.
It follows from a slight modification of the argument in \cite{Forlano}
that 
\eqref{BBM6} is locally well-posed for $\al < \frac 34$.
Moreover, variance blowup occurs
for $\al \ge \frac 34$.

Let  $\al \ge \frac 34$.
By following Hairer's work \cite{Hairer}, 
we propose to consider the following renormalized stochastic BBM:
\begin{align}
\dt u_N-\partial_{txx}u_N+\partial_{x}u_N
+\partial_{x}(u_N^{2})  = C_{1-\al, N}  \P_N \jb{\dx}^\al  \dx \xi,   
\label{BBM7}
\end{align}

\noi
where 
$C_{1-\al, N}$ is as in \eqref{CN}.
Then, by a straightforward modification 
of the proof of Theorem~\ref{THM:3}, 
we can  show that 
the solutions $u_N$ to \eqref{BBM7}
(with an appropriate assumption on deterministic initial data)
converge in law to the solution $u$ to 
the stochastic BBM forced by 
a space-time noise
$\upze$ whose time marginal 
$\upze(t)$
is given by  a scaled spatial white noise 
for each $t \in \R_+$;
see Theorem \ref{THM:4a}.
See Appendix \ref{SEC:A}
for a further discussion, 
where we also discuss renormalization
by taming the nonlinearity as \eqref{BBM8},
leading to an analogue of Theorem~\ref{THM:3a};
see Theorem \ref{THM:6}.

\end{remark}

\begin{remark}\label{REM:DH} \rm

In Theorem \ref{THM:3a}, 
we considered the weakly interacting BBM \eqref{BBM8}
and obtained the non-trivial limiting dynamics  \eqref{BBM9};
while the equation \eqref{BBM9} is linear, 
the stochastic forcing $\dx \ze$
on the right-hand side captures the limiting behavior of the nonlinear
interaction described by the second Picard  iterate $Z_N $ in \eqref{I3}.
In the following, we provide two examples
of the study of weakly nonlinear dynamics with random initial data
where one may find some similarity. 

\medskip

In a series of breakthrough works \cite{DH1, DH2, DH3, DH5}, 
Deng and Hani studied
the rigorous derivation of the wave kinetic equation, 
where one studies
the limiting behavior of weakly interacting nonlinear dynamics
of strength $\epsilon$ 
on the dilated torus of size $L$
as $\epsilon \to 0$ and $L \to \infty$.
In this context, the limiting dynamics is described by 
the resonant dynamics.
We refer interested readers to the survey papers \cite{Deng, DH4}.
See also the works by Collot and Germain.

\medskip

In \cite{GOTT}, 
Greco, Tao, Tolomeo, and the third author 
considered the construction of the Gibbs measure 
for the focusing cubic  NLS on $\T^2$:
\begin{align}
i \dt u + (1- \Dl)   u -  |u|^{2} u
= 0.
\label{NLS1}
\end{align}

\noi
It is known that the Gibbs measure for \eqref{NLS1}
does not exist (see \cite{BS, OST})
and thus the authors in \cite{GOTT} investigated the problem
in the weakly interacting regime:
\begin{align}
i \dt u_N + (1 - \Dl) u_N -  C_N \P_N(|\P_N u_N|^{2} \P_N u_N)
= 0, 
\label{NLS2}
\end{align}

\noi
where $C_N > 0$ tends to $0$ as $N \to \infty$.
Let $\rho_N$ be the frequency-truncated Gibbs measure 
for~\eqref{NLS2}
with a renormalized $L^2$-cutoff:
\begin{align}
d\rho_N(u)=Z_{N}^{-1}
\ind_{\{|\int_{\T^2} \, : |\P_N u|^2: \, dx| \, \leq K_N\}}  
e^{ C_N \int_{\T^2} \, : |\P_N u|^4: \,dx}d\mu(u),
\label{NLS3}
\end{align}

\noi
where $\mu$ denotes the (massive) Gaussian free field on $\T^2$.
Here, $K_N > 0$ denotes 
 the cutoff size which, for simplicity,  we assume to diverge as $N \to \infty$ in the following. 
Here, $: \! |\P_N u|^{2k}\!:$, $k = 1, 2$, 
denote the Wick renormalized powers; see \cite{GOTT}
for a precise definition.
Then, at the dynamical level, 
it was shown in \cite{GOTT}
that there exists $C_* > 0$
such that 
if 
$C_N \le C_* (K_N + \log N)^{-1}$
for any $N \in \N$, 
then
the solution $u_N$ to
\eqref{NLS2} 
converges to the solution $u$ to the linear Schr\"odinger equation
$i\dt u + (1-\Dl)u=0$
with the initial data distributed by 
the  Gaussian free field $\mu$ on $\T^2$
(which is an invariant measure for the linear dynamics).
See \cite[Remark 1.2]{GOTT}
for a further discussion.

\end{remark}

\begin{remark} \label{REM:NLW} \rm
Consider the quadratic nonlinear wave equation (NLW)
on $\T^2$:\footnote{In this formal discussion, we ignore the issue of renormalization.}
\begin{align}
\begin{cases}
\dt^2 u + (1- \Dl) u = u^2\\
(u, \dt u)|_{t = 0} = (u_0, u_1), 
\end{cases}
\label{NLW1}
\end{align}

\noi
where   the random initial data 
$(u_0, u_1) = (u_0^\o(\al), u_1^\o(\al)) $
is given by\footnote{Here, we use a different convention on $\al$ than that in \eqref{data1}
to match the notation in \cite{OO}.} 
\begin{align*}
(u_0, u_1)   =
\bigg( \sum_{n\in \Z^2} \frac{g_n(\o) }{\jb{n}^{1-\al}} e_n,
\sum_{n\in \Z^2} \jb{n}^\al h_n(\o)  e_n \bigg).
\end{align*}

\noi
Here, 
$\{g_n, h_n \}_{n\in \Z^2}$ is a family  of independent 
complex-valued standard Gaussian random variables
conditioned that $g_{-n} = \overline{g_n}$
and $h_{-n} = \overline{h_n}$, $n \in \Z$.
Then, a straightforward modification of the 
argument in \cite{OO}
shows that 
variance blowup on the second Picard iterate
occurs
for  $\al \ge \frac 12$.
We will address the issue of renormalization 
beyond variance blowup
for \eqref{NLW1}
in a forthcoming work \cite{LLLOT}, 
where we crucially exploit multilinear dispersive properties.

\end{remark}

%
%
%
%
%
%
%
%
%

\section{Preliminaries}
\label{SEC:2}

\subsection{Notations}

Let $A\les B$ denote an estimate of the form $A\leq CB$ for some constant $C>0$. We write $A\sim B$ if $A\les B$ and $B\les A$, while $A\ll B$ denotes $A\leq c B$ for some small constant $c> 0$. 
We may write  $\les_{\be}$ to 
emphasize the dependence on an external parameter $\be$.
We use $C>0$ to denote various constants, which may vary line by line.
We may also write $C_\be$ to emphasize 
the dependence on a parameter $\be$.

We set $\Z^* = \Z\setminus\{0\}$.
We also  set
$a \wedge b = \min(a, b)$.

In this paper, all the function spaces
are restricted to real-valued functions.
Given $s \in \R$, 
we define the $L^2$-based Sobolev space $H^{s}(\T)$  
via the norm: 
\begin{align*}
\|  f  \|_{H^{s}}
 =
\|\jb{n}^{s} \ft f( n)\|_{\l^2_n}.
\end{align*}

\noi
Given $s \in \R$ and $1\le p\le \infty$,
we define 
 the $L^p$-based Sobolev space $W^{s,p}(\T)$
as the completion of $\D(\T) = C^\infty(\T)$
under 
 the norm:
\begin{align*}
\| f\|_{W^{s,p}}
= \|\jb{\nb}^s f\|_{L^p}
= \|\F^{-1} (\jb{n}^s \ft f (n))\|_{L^p}
\end{align*}

\noi
We note that $H^s(\T) = W^{s,2}(\T)$.
When $p = \infty$, our definition 
provides a smaller space than the usual $W^{s, \infty}(\T)$, 
but it ensures that our version of 
$W^{s, \infty}(\T)$ is separable; see Remark \ref{REM:Sk2}.
We note that such a convention is common in 
(stochastic) parabolic PDEs; see, for example, 
\cite{MW1, COW}. 
We endow 
$\D(\T) = C^\infty(\T)$
 with the Fr\'echet space topology generated by a countable
family of semi-norms
$\big\{\|\,\cdot \,\|_{W^{j, \infty}}\big\}_{j \in \N\cup\{0\}}$.
We use $\D'(\T)$ to denote the space of distributions on~$\T$.

In dealing with space-time functions, 
we use  short-hand notations such as  
$C_TH^s_x$ = $C([0, T]; H^s(\T))$
and  
$C_I H^s_x = C(I; H^s(\T))$, where $I\subset \R_+$ is an interval.

Given $0 < \be < 1$, we denote by $C^\be_T B = C^\be([0, T]; B)$ the space of $\be$-H\"older continuous functions taking values in a Banach space $B$, endowed  with the seminorm:
\begin{align}
\|f\|_{C^\be_T B} = \sup_{0 \le r < t \le T}
\frac{\|u(t) - u(r)\|_B}{|t-r|^\be}.
\label{HO1}
\end{align}

\noi
We also define the Lipschitz space
 $\C^\be_T B = \C^\be([0, T]; B)$ via the norm:
\begin{align}
\| f \|_{\C^\be_TB} = \| f \|_{L^\infty_TB} + \|u\|_{C^\be_T B}.
\label{HO2}
\end{align}

Given $N \in \N$, 
we use   $\P_N$ to denote
 the Dirichlet projector onto the frequencies 
$\{|n| \leq N\}$
defined by 
\begin{align}
 \ft{\P_N f}(n)= \ind_{\{|n|\leq N\}} \ft f(n).
 \label{Diri1}
\end{align}

\subsection{Deterministic tools}

We first recall the key product estimate
in studying the BBM equation; see \cite[Lemma 1]{BT09}
and \cite[Lemma 3.1]{RO10}.

\begin{lemma}\label{LEM:BT}
Let $s\ge 0$.
Then, we have
\begin{align*}
\|\vp(D)(fg)\|_{H^s(\T)} \les \|f \|_{H^s(\T)}\|g \|_{H^s(\T)}, 
\end{align*}

\noi
where $\vp(D)$ is as in \eqref{phi1}.

\end{lemma}

Next, we recall the product estimate on 
the product of functions of negative and positive regularities;
see \cite[Proposition 2]{BOZ25}.
See also \cite[Lemma 3.4]{GKO}.

\begin{lemma}\label{LEM:BOZ}
Let $s > 0$,  $1 < p \le \infty$ and $1 < q, r < \infty$ such that 
\[ \frac 1p + \frac 1q \le \frac 1r + \frac sd
\qquad \text{and}\qquad
q, r' \ge p', \] 

\noi
where $p'$ and $r'$ denote the H\"older conjugates of $p$ and $r$, respectively.
Then, we have
\begin{align*}
\| \jb{\nb}^{-s} (fg)\|_{L^r(\T^d)}
\les \| \jb{\nb}^{-s} f\|_{L^p(\T^d)}
\| \jb{\nb}^s g\|_{L^q(\T^d)}.
\end{align*}

\end{lemma}

Lastly, we recall the following elementary calculus lemma;
see, for example, \cite[Lemma~4.2]{GTV}
(in the continuous setting)
and \cite[Lemma 2.1]{TzvBO}.

\begin{lemma}\label{LEM:SUM}
Let $\be \ge \g \ge 0$ and $\be + \g > 1$.
Then, 
we have
\begin{align*}
\sum_{n_1 \in \Z} \frac{1}{\jb{n_1}^\be\jb{n - n_1}^\g}\les \frac{\phi_\be(n)}{\jb{n}^\g}, 
\end{align*}

\noi
where $\phi_\be(n)$ is given by 
\begin{align*}
\phi_\be(n) = \sum_{|n_1| \le |n|} \frac{1}{\jb{n_1}^\be} \sim
\begin{cases}
1, & \text{if } \be > 1, \\
\log (2 + |n|), & \text{if } \be = 1, \\
\jb{n}^{1- \be}, & \text{if } \be < 1.
\end{cases}
\end{align*}

\end{lemma}

\subsection{Probabilistic tools}

Let $\{g_n\}_{n\in \Z}$ be a sequence of 
independent 
complex-valued standard Gaussian random variables 
on some probability space 
$(\Omega, \sF, \bP)$ 
conditioned that 
$g_{-n} = \cj{g_n}$, $n \in \Z$.
Then, it is known that 
\begin{align}
\E [g_n^{k} \cj{g_m^\l}] = \ind_{n= m} \ind_{k= \l}\cdot  k!
\label{mom1}
\end{align}

\noi
for any $n, m \in \Z^*$ and $k, \l \in \N$.
In particular, we have 
\begin{align}
\E\big[|g_n|^2\big] =1 \ \  \text{for any $n \in \Z$}
\qquad 
\text{and} \qquad \E\big[|g_n|^4\big] =
\begin{cases}
2 
&  \text{for $n \in \Z^*$}\\
3 &  \text{for $n = 0$}.
\end{cases}
\label{mom2}
\end{align}

\noi

\medskip

Next, we recall a useful tool to prove
weak convergence of distribution-valued random variables, following~\cite{BDW}.
While the main results in 
\cite{BDW}
are stated for $\S'(\R^d)$-valued random variables, 
where $\S'(\R^d)$ denotes the space of  tempered distributions on $\R^d$, 
the same results also hold for
 $\D'(\T^d)$-valued (and, in fact, $(\D'(\T^d))^{\otimes m}$-valued) random variables
on $\T^d$, 
since the proof is based on the sequential representation
of elements in the Schwartz class $\S(\R^d)$
and $\S'(\R^d)$ (via Fourier-Hermite series; see \cite[Section 3]{BDW})
which also holds for $\D(\T^d)$
and $\D'(\T^d)$ (via Fourier series).

The strong topology on $\D'(\T^d)$ is generated 
by the family of semi-norms 
\begin{align*}
q_B(X) = \sup_{f\in B}{|\jb{X,f}|}, 
\quad B\subset \D(\T^d) {\text{ bounded}},
\ X\in \D'(\T^d),
\end{align*}

\noi
where a subset $B\subset \D(\T^d)$ is bounded 
if, for all neighborhood $V$ of $0$, there exists
$\ld>0$ such that $B\subset \ld V$. 
Given an integer $m \ge 2$, 
by the strong topology on 
$(\D'(\T^d))^{\otimes m}$, 
we mean the product topology 
of the strong topology on each factor.

Fix $m \in \N$.
We say that 
a sequence $\{X_n\}_{n \in \N}$
of $(\D'(\T^d))^{\otimes m}$-valued
random variables
 on some probability
space $(\Omega, \sF, \bP)$
converges in law to $X$
with respect to the strong topology 
of $(\D'(\T^d))^{\otimes m}$, 
if 
\begin{align}
\lim_{n\to \infty} \int_{(\D'(\T^d))^{\otimes m}} F(\phi) d\mu_{X_n}(\phi)= 
 \int_{(\D'(\T^d))^{\otimes m}} F(\phi) d\mu_X(\phi)
\label{BDW1}
\end{align}

\noi
for any function $F$ that is continuous and bounded on 
$(\D'(\T^d))^{\otimes m}$ equipped with the strong topology, 
where
$\mu_{X_n}$
and  $\mu_X$ denote the laws of $X_n$ and $X$, respectively, given by 
\begin{align*}
\mu_{X_n} = \bP\circ X_n^{-1}\qquad\text{and}
\qquad 
\mu_X = \bP\circ X^{-1}.
\end{align*}

\noi
The following lemma shows that 
 convergence in law 
 with respect to the strong topology~\eqref{BDW1}
is equivalent to 
convergence of  $\jb{X_n, \psi}_m
\to \jb{X, \psi}_m$ in law for every test function $\psi
= (\psi_1, \dots, \psi_m)\in (\D(\T^d))^{\otimes m}$.
Here, 
$\jb{f, g}_m$ 
denotes the $(\D'(\T^d))^{\otimes m}$-$(\D(\T^d))^{\otimes m}$ duality pairing
given by 
\begin{align}
\jb{f, g}_m = ( \jb{f_1, g_1}, \dots, \jb{f_m, g_m})
\label{dual1}
\end{align}

\noi
for $f = (f_1, \dots, f_m) \in
(\D'(\T^d))^{\otimes m}$
and 
$g = (g_1, \dots, g_m) \in
(\D(\T^d))^{\otimes m}$, 
where $\jb{f_j, g_j}$
denotes 
the $\D'(\T)$-$\D(\T)$ duality pairing.

\begin{lemma}
\label{LEM:BDW}
Let $m \in \N$.
A sequence
  $\{X_n\}_{n \in \N}$
of $(\D'(\T^d))^{\otimes m}$-valued random variables
converges in law to $X$
with respect to the strong topology 
of $(\D'(\T^d))^{\otimes m}$
if and only if 
$\R^m$-valued random variables 
$\jb{X_n, \psi}$
converges in law to   
$\jb{X, \psi}$
for all 
 $\psi\in (\D(\T^d))^{\otimes m}$.

\end{lemma}

See \cite[Th\'eor\`eme III.6.5]{Fer1}
and  \cite[Theorem 5.1]{BDW} for a proof
when $m = 1$; see also 
\cite[Theorem 2.2]{Simon0}, 
\cite[comment after Theorem 2.3]{BDW}, 
and \cite[Corollary~2.4]{BDW}.
We note that the argument in \cite{BDW}
easily extends to the case of general $m \in \N$.
In Lemma \ref{LEM:BDW}, 
we use the strong topology
but the same result also holds
with the weak topology 
of $(\D'(\T^d))^{\otimes m}$,
since
 $(\D'(\T^d))^{\otimes m}$ is a Montel space
and thus any weakly convergent sequence
is also strongly convergent
thanks to the Banach-Steinhaus theorem 
for
Fr\'echet spaces;
see \cite[Remark 6.6]{BDW}
and \cite[Corollary 1 of Proposition 34.6]{Treves}.

In Subsection \ref{SUBSEC:Z1}, 
we study convergence of finite-dimensional 
marginals 
$\{ Z_N(t_j)\}_{j = 1}^m$
where  $Z_N$ is as in \eqref{I3}.
By viewing it as a $ (\D'(\T))^{\otimes m}$-valued random variable, 
we can 
apply Lemma~\ref{LEM:BDW}
and reduce the matter to establishing 
convergence of 
$\big\{ \jb{Z_N(t_j), \psi_j}\big\}_{j = 1}^m$
for each $\{\psi_j\}_{j = 1}^m
\in (\D(\T^d))^{\otimes m}$, 
which we identify with 
\[\Big\langle\bigotimes_{j = 1}^m Z_N(t_j),  \bigotimes_{j = 1}^m \psi_j\Big\rangle_m.\]


\noi
For this purpose, we  mention
a related useful result
\cite[Theorem~6.15]{Walsh}.

\medskip

Next, we  review several basic concepts in 
Gaussian analysis; see, for example, 
\cite{Janson, Nualart, NP} for further details.

An isonormal Gaussian process 
$W=\{W(f): f\in L^2(\T)\}$
over $L^2(\T)$ on some probability space 
$(\Omega, \sF, \bP)$
is a centered Gaussian process 
 such 
that
\begin{align*}
\E[W(f_1)W(f_2)] = \jb{f_1,f_2}_{L^2(\T)}
\end{align*}

\noi
for any $f_1, f_2 \in L^2(\T)$.
For our application, 
 set 
\begin{align}
W(f) = \sum_{n\in \Z} \ft f (n) g_n, 
\label{WW1}
\end{align}

\noi
 for $f\in L^2(\T)$,
where  $\{g_n\}_{n \in \Z}$
is  as in \eqref{data1}.
Note that, given any $f \in L^2(\T)$,  
 the sum in~\eqref{WW1} converges in $L^2(\Omega, \sF, \bP)$.
The action~\eqref{WW1} on $h$ by the white noise
is referred to as the white noise functional in \cite{OTh1, ORT}.
Note that $W(f)$ is basically the `periodic' Wiener integral on $\T$.

Given $k \in \N \cup \{0\}$, 
 the Hermite polynomials $H_k(x)$
 of degree $k$ is defined
 through the generating function:
\begin{equation*}
  e^{tx - \frac{1}{2} t^2} = \sum_{k = 0}^\infty \frac{t^k}{k!} H_k(x)
 \end{equation*}
	
\noi
for $t, x \in \R$.
 For readers' convenience, we write out the first few Hermite polynomials
in the following:
\begin{align*}
& H_0(x) = 1, 
\qquad 
H_1(x) = x, 
\qquad
H_2(x) = x^2 - 1. \\
\end{align*}

We use 
 $\H_k$ to denote the $k$th
homogeneous Wiener chaos, 
which is the closed linear subspace 
of $L^2(\Omega, \sF, \bP)$
generated by the set 
$\{H_k(W(f)): f\in L^2(\T), \|f\|_{L^2(\T)}=1\}$.
We then set 
\[ \H_{\le k} =  \bigoplus_{j = 0}^k \H_{j}.\]

\noi
For $k \in \N$
and $f\in L^2(\T)$ 
with $\|f\|_{L^2(\T)}=1$,
the mapping
$I_k$ satisfying 
\begin{align}
I_k(f^{\otimes k}) = H_k(W(f))
\label{WW2}
\end{align}

\noi
can be extended to a linear isometry 
between $L^2_\text{sym}(\T^k)$ 
and $\H_k$, 
where $L^2_\text{sym}(\T^k)$
is the closed subspace of 
$L^2(\T^k)$ 
consisting of symmetric functions, 
equipped with the scaled norm:
$\|f\|_{L^2_\text{sym}(\T^k)}
= \sqrt{k!} \|f\|_{L^2(\T^k)}$.
For general 
$f\in L^2(\T^k)$, 
we can extend the definition of $I_k$
by setting 
\begin{align}
I_k(f) = I_k(f_\text{sym} ),
\label{WW3}
\end{align}

\noi
where 
$f_\text{sym}$ is the symmetrization of $f$ defined by 
\begin{align}
f_\text{sym} (x_1, x_2, \cdots, x_k)
= \frac{1}{k!} \sum_{\s\in S_k}
f(x_{\s(1)}, x_{\s(2)}, \cdots, x_{\s(k)}).
\label{WW4}
\end{align}

\noi
Here,  $S_k$ is the symmetric group on 
$\{1, 2, \cdots, k\}$.
We note that 
$I_k(f)$ corresponds to  the multiple Wiener integral
with respect to $W$; see 
\cite[Section 1.1.2]{Nualart}.

Given  $f_j\in L^2_\text{sym}(\T^{k_j})$, $j = 1, 2$, 
and $r=1,\cdots, \min(k_1,k_2)$, the contraction of $r$
indices of $f_1$ and $f_2$ is an element of 
$L^2(\T^{k_1+k_2-2r})$ defined by 
\begin{align}
\begin{aligned}
f_1 \otimes_r f_2 (x_1, \cdots, x_{k_1+k_2-2r})
& = \int_{\T^r}
f_1(x_1, \cdots, x_{k_1-r}, y_1, \cdots, y_r) \\
& \qquad \times
f_2(x_{k_1 +1 - r }, \cdots, x_{k_1+k_2-2r}, y_1, \cdots, y_r)
dy_1\cdots dy_r,
\end{aligned}
\label{N3a}
\end{align}

\noi
and $f_1 \otimes_0 f_2 = f_1\otimes f_2$.

As a consequence
of the  hypercontractivity of the Ornstein-Uhlenbeck
semigroup  due to Nelson \cite{Nelson2}, 
we have the following Wiener chaos estimate
\cite[Theorem~I.22]{Simon}.
See also \cite[Proposition~2.4]{TTz}.

\begin{lemma}\label{LEM:hyp}
Let $k \in \N$.
Then, we have
\begin{equation*}
\|X \|_{L^p(\O)} \leq (p-1)^\frac{k}{2} \|X\|_{L^2(\O)}
 \end{equation*}
 
 \noi
 for any finite $p \geq 2$
 and any $X \in \H_{\le k}$

\end{lemma}

The following lemma allows us to determine
regularity properties
of a given process.
For $h\in \R$, we define the difference operator
by 
\begin{align}
\dl_h X(t) = X(t+h) - X(t).
\label{dl1}
\end{align}

\noi
We say a stochastic process $X:\R_+ \to \D'(\T^d)$
is stationary in space if for any $x\in \T^d$, 
the processes $\{X(t, \cdot)\}_{t\ge 0}$ 
and $\{X(t, x+\cdot)\}_{t\ge 0}$ have the same law.

\begin{lemma}
\label{LEM:OOT}
Let 
$\{ X_N \}_{N \in \N}$
and $X$ 
spatially stationary  stochastic processes:
$\R_+\to \D'(\T^d)$. 
Suppose that there exists $k\in \N$ such that 
$X_N(t)$ and  $X(t)$ 
belong to $\H_{\le k}$ for any $t\in \R_+$.

\smallskip
\noi
{\rm (i)}
Let $t\in \R_+$. If there exists $s_0\in \R$ such that 
\begin{align*}
\E\big[|\ft{X}(t,n)|^2\big] \les \jb{n}^{-d-2s_0}
\end{align*}

\noi
for any $n\in \Z^d$,
then we have $X(t) \in  W^{s,\infty}(\T^d)$  for $s<s_0$,
 almost surely, satisfying
\begin{align}
\E\big[\| X(t)\|_{W^{s, \infty}}^p\big] \les p^{\frac{kp}{2}}
\label{dl2}
\end{align}

\noi
for any $p \ge 1$.
In particular, we have 
\begin{align}
\PP \Big(  \| X(t)\|_{W^{s, \infty}}
> \ld \Big)
\le C e^{-c\ld^\frac2k}
\label{dl3}
\end{align}

\noi
for any $\ld > 0$.
Furthermore, if there exists $\g > 0$ such that 
\begin{align*}
\E\big[ |\ft X_N(t, n) - \ft X(t, n)|^2\big]\les N^{-\g} \jb{n}^{ - d - 2s_0}
\end{align*}

\noi
for any $n \in \Z^d$ and $N \geq 1$, 
then 
$X_N(t)$ converges to $X(t)$ in $W^{s, \infty}(\T^d)$, $s < s_0$, 
almost surely.

\smallskip
\noi
{\rm(ii)}
Let $T > 0$ and suppose that \textup{(i)} holds on $[0, T]$.
If there exists $\theta \in (0, 1)$ such that\footnote{There is a typo in 
 \cite[Proposition 2.7\,(ii)]{OOT};
in order to match the result
in \cite[Lemma A.1]{OOT}, 
$\ta$ in \cite[Proposition 2.7\,(ii)]{OOT}
needs to be replaced by $ 2\ta$.} 
\begin{align*}
\E\big[ |\dl_h \ft{X}(t,n)|^2\big]
 \les \jb{n}^{ - d - 2s_0}
|h|^{2\theta}
\end{align*}

\noi
for any  $n \in \Z^d$, $t \in [0, T]$, and $h \in [-1, 1]$,\footnote{We impose $h \geq - t$ such that $t + h \geq 0$.}
then we have 
$X \in C([0, T]; W^{s, \infty}(\T^d))$, 
$s < s_0$,  almost surely, 
satisfying
\begin{align}
\E\big[\| \dl_h X(t)\|_{W^{s, \infty}}^p\big] \les p^{\frac{kp}{2}}|h|^{p \ta}
\label{dl4}
\end{align}

\noi
for any $p \ge 1$.
In particular,  given any $p > \ta^{-1}$ and $0 < \be < \ta - \frac 1p$, 
there exists $C_{p, \be} > 0$ such that 
\begin{align}
\PP \Big(  \| X\|_{C^\be(I; W^{s, \infty}_x)}
> \ld \Big)
\le C_{\be,  p} \ld^{-p} 
\label{dl5}
\end{align}

\noi
for any $\ld > 0$ and any subinterval $I \in [0, T]$ with $|I| \le 1$, 
where $C^\be(I; W^{s, \infty}(\T))$ is as in~\eqref{HO1}.
Furthermore, 
if there exists $\g > 0$ such that 
\begin{align*}
 \E\big[ |\dl_h \ft X_N(n, t) - \dl_h \ft X(n, t)|^2\big]
 \les N^{-\g}\jb{n}^{ - d - 2s_0}
|h|^{2\ta}
\end{align*}

\noi
for any  $n \in \Z^d$, $t \in [0, T]$,  $h \in [-1, 1]$, and $N \geq 1$, 
then 
$X_N$ converges to $X$ in $C([0, T]; W^{s, \infty}(\T^d))$, $s < s_0$,
almost surely.

\smallskip
\noi
{\rm(iii)}
Suppose that the hypotheses in {\rm (i)} for $t = 0$ and {\rm (ii)} hold.
 Given any $p > \ta^{-1}$ and $0 < \be < \ta - \frac 1p$, 
there exists $C_{p, \be} > 0$ such that 
\begin{align}
\PP \Big(  \| X\|_{\C^\be_T W^{s, \infty}_x}
> \ld \Big)
\le C_{\be,  p} T \ld^{-p} 
\label{dl6}
\end{align}

\noi
for any $\ld > 0$, 
where $\C^\be([0, T]; W^{s, \infty}(\T))$ is as in \eqref{HO2}.

\end{lemma}

\begin{proof}
Most of (i) and (ii) follows from 
in \cite[Proposition~2.7 and Lemma~A.1]{OOT}
whose proof is based on the Wiener chaos estimate (Lemma \ref{LEM:hyp})
and Kolmogorov's continuity criterion type argument
(for (ii));
see also \cite[Proposition~3.6]{MWX}.
The bound~\eqref{dl3} follows from \eqref{dl2}
and \cite[Lemma 4.5]{TzvBO}
(or Chebyshev's inequality as in \cite[Lemma 3]{BOP1}\footnote{Lemma 2.2 in the arXiv version.}), 
while \eqref{dl5}
follows from~\eqref{dl4}
and  Kolmogorov's continuity criterion
\cite[Exercise~8.2]{Bass} (which also holds
for a general metric space).
Lastly, \eqref{dl6}
follows from \eqref{dl3}, \eqref{dl5}, and the subadditivity over (disjoint) intervals.
\end{proof}

We now recall a version of the 
 fourth moment theorem
(see 
\cite[Theorem 6.2.3]{NP} with
\cite[Theorem 5.2.7]{NP}).
%
We use the following lemma to 
prove 
convergence of  finite-dimensional marginals (in time)
of $Z_N$ in~\eqref{I3} 
(more precisely, tested against test functions
in the form $\big\{ \jb{Z_N(t_j), \psi_j}\big\}_{j = 1}^m$),
showing uniqueness of the limit
of $\{Z_N \}_{N \in \N}$.
See Subsection \ref{SUBSEC:Z1}.

\begin{lemma}[fourth moment theorem]
\label{LEM:FMT}
Let   $m \in \N$
and $\{k_i\}_{i = 1}^m  \subset \N$.
Let $\{F_n\}_{n \in \N}$ be a sequence of random vectors given by 
\begin{align*}
F_n = (F_n^1,\cdots, F_n^m) = (I_{k_1}(f^1_n), \cdots, I_{k_m}(f^m_n)), 
\end{align*}

\noi
where 
$f^i_n\in L^2_\textup{sym}(\T^{k_i})$, $i = 1, \dots, m$, 
such that 
an $m\times m$  matrix  $\Cs=(\Cs_{ij})_{1\le i,j\le m}$ 
defined by 
\noi
\begin{align}
\Cs_{ij}
= 
\lim_{n\to \infty} \E[F_n^i F_n^j] , \quad 1\le i,j\le m, 
\label{I5}
\end{align}

\noi
is positive semi-definite.
Then, 
the following statements are equivalent\textup{:}

\smallskip
\noi

\begin{itemize}
\item
[\textup{(i)}]
For every $1\le i \le m$, we have
\begin{align*}
\lim_{n\to \infty} \E[(F_n^i)^4]  = 3 \Cs_{ii}^2. 
\end{align*}

\smallskip
\item[\textup{(ii)}]
For every $1\le i \le m$ and  $1\le r\le k_i -1$, we have 
\begin{align*}
\lim_{n\to \infty}\|f_n^i \otimes_r f_n^i\|_{L^2(\T^{2(k_i-r)})}
=0.
\end{align*}

\smallskip

\item[\textup{(iii)}]
As $n \to \infty$, the random vector $F_n$ converges in law to an $m$-dimensional Gaussian vector
$Z\sim \NN(0,\Cs)$.

\end{itemize}

\end{lemma}

In this paper, we only use the implication  (ii) $\LRA$ (iii)
but we included the  statement~(i) since it 
is the reason that Lemma \ref{LEM:FMT}
is called the fourth moment method.
We also point out that, while 
positive definiteness of the matrix $\Cs$ is imposed in 
some literature (for example, 
\cite[Proposition~1]{PT}), 
it suffices to assume
that $\Cs$ is 
positive semi-definite.

The fourth moment theorem 
 was first introduced in \cite{NP05}
 by Nualart and Peccati
 and was  further extended in \cite{NO08, NP09}, 
 which has  become a powerful tool in probability theory, 
particularly in applications to the study of scaling limits of disordered systems 
such as directed polymers, SPDEs, SDEs, 
and central limit theorems for Gaussian functionals; 
see, for example, \cite{CSZ17, NZ20, NPR10, HNVZ20, Hairer, GT25}.

\medskip

Lastly, 
we recall the Prokhorov theorem
and the Skorokhod representation theorem.

\begin{definition}\label{DEF:tight}\rm
Let $\J$ be any nonempty index set.
A family 
 $\{ \mu_j \}_{j \in \J}$ of probability measures
on a metric space $\M$ is said to be   tight
if, for every $\dl > 0$, there exists a compact set $K_\dl\subset \M$
such that $ \sup_{j\in \J}\mu_j(K_\dl^c) < \dl$. 
We say  that $\{ \mu_j \}_{j \in \J}$ is relatively compact, if every sequence 
in  $\{ \mu_j \}_{j \in \J}$ contains a weakly convergent subsequence.

We say that a family 
 $\{ X_j \}_{j \in \J}$ of random variables
 is tight if their laws are tight.

\end{definition}

Note that the index set $\J$ does not need to be countable.
We now 
recall the following 
Prokhorov theorem from \cite{PB71, Bass}.

\begin{lemma}[Prokhorov theorem]
\label{LEM:Pro}

If a sequence of probability measures 
on a metric space $\M$ is tight, then
it is relatively compact. 
If in addition, $\M$ is separable and complete, then relative compactness is 
equivalent to tightness.

\end{lemma}

Lastly, we recall  the following Skorokhod representation 
theorem from  
\cite[Chapter 31]{Bass}.

\begin{lemma}[Skorokhod representation theorem]\label{LEM:Sk}

Let $\M$ be a complete  
separable metric space \textup{(}i.e.~a Polish space\textup{)}.
Suppose that    
 probability measures 
 $\{\mu_n\}_{n\in\N}$
 on $\M$ converges  weakly   to a probability measure $\mu$
as $n \to\infty$.
Then, there exist a probability space $(\wt \O, \wt \sF, \wt\PP)$,
and random variables $X_n, X:\wt \O \to \M$ 
such that 
\begin{align*}
\Law( X_n) = \mu_n
\qquad \text{and}\qquad
\Law(X) = \mu ,
\end{align*}

\noi
and $X_n$ converges $\wt\PP$-almost surely to $X$ as $n\to\infty$.
Here, $\Law (X)$ denotes the law of a random variable $X$.

\end{lemma}

\begin{remark}\label{REM:Sk2} \rm 
Recall that 
 the space of continuous functions from 
a  separable metric space 
 to another separable metric space 
 with the compact-open  topology (in time) is separable; see \cite{Mi}.
 See also the paper \cite[Corollary 3.3]{Khan86}. 
In particular, 
in view of our definition of $W^{s, \infty}(\T)$ which makes the space separable, 
we see that $C(\R_+; W^{s, \infty}(\T))$ endowed with the compact-open 
topology is separable.

\end{remark}

\section{Convergence of the second Picard iterates}
\label{SEC:Z}

In this section, we present a proof of Theorem \ref{THM:1}
on convergence in law of the second Picard iterate $Z_N$ in~\eqref{I3} to $Z$ in \eqref{lin2}.
In Subsection \ref{SUBSEC:Z1}, 
we prove uniqueness of a possible limit
by using the fourth moment theorem (Lemma \ref{LEM:FMT}).
In Subsection \ref{SUBSEC:Z2}, 
 we then establish tightness of 
the sequence $\{Z_N\}_{N \in \N}$  in 
 $C(\R_+;W^{s,\infty}(\T))$
 and prove 
Theorem \ref{THM:1}.

\subsection{Uniqueness of the limit}
\label{SUBSEC:Z1}

Our main goal in this subsection is 
to establish the following proposition
on convergence of finite-dimensional marginals (in time)
of $Z_N$, which shows the uniqueness of the limit
of $\{Z_N \}_{N \in \N}$ (if it exists).

\begin{proposition}\label{PROP:uniq}
Let $\al\le \frac 14$.
Let $Z_N$ and $Z$ be as in \eqref{I3} and \eqref{lin2}, respectively.
Given $m \in \N$, let 
 $0\le t_1\le  \cdots \le t_m$.
Then, 
$\big\{Z_N(t_j)\big\}_{j = 1}^m$
converges in law
to 
$\big\{Z(t_j)\big\}_{j = 1}^m$
in $(\D'(\T))^{\otimes m}$
as $N \to \infty$.

\end{proposition}

As a preliminary step, 
we study the covariance (in time)
of $Z_N(t)$ acting on  test functions.

\begin{lemma}
\label{LEM:COV}

Let $\al \le \frac 14$. 
Given $N \in \N$, 
let  $Z_N$ and $Z$ be  as in \eqref{I3} and \eqref{lin2}, respectively.
Then, 
given any test function $\psi \in \D(\T)$, we have 
\begin{align}
\E\big[\jb{Z_N (t), \psi}\big] = \E\big[\jb{Z (t), \psi}\big] =0
\label{con1}
\end{align}

\noi
for any $t \ge 0$ and $N \in \N$, 
and 
\begin{align}
\lim_{N\to \infty} 
\E\Big[\jb{Z_N (r), \psi_1} \jb{ Z_N (t), \psi_2}\Big]
= 
\E\Big[\jb{ Z (r), \psi_1}\jb{ Z (t), \psi_2}\Big]
\label{con2}
\end{align}

\noi
\noi
for any test functions $\psi_1, \psi_2 \in \D(\T)$ and $t \ge r \ge 0$, 
where 
$\jb{f, g}$ 
denotes the $\D'(\T)$-$\D(\T)$ duality pairing.


\end{lemma}

\begin{proof}

From \eqref{I3} with \eqref{lin1}, we have  
\begin{align}
\begin{aligned}
\ft Z_N (t,n)
& =  
\sum_{n= n_1 + n_2} 
\int^t_0 
e^{(t-t')\vp(n)}\vp(n)
\ft z_N (t', n_1) \ft z_N (t', n_2) dt'\\
& =   C_{\al, N}^2
 e^{t \vp(n)}
 \sum_{\substack{n= n_1 + n_2 \\ |n_1|, |n_2|\leq N}}
 \int^t_0
 e^{- t' (\vp(n) -\vp(n_1) -\vp(n_2))} 
 \vp(n)
\frac{g_{n_1} g_{n_2}}{\jb{n_1}^\al \jb{n_2}^{\al}}
 dt'.
\end{aligned}
\label{con3}
\end{align}

\noi
Recalling that $\vp(n)|_{n = 0} = 0$, 
we have $\ft Z_N(t, 0) =  0$
for any $t \in \R_+$
(when $n = 0$).
Then, from~\eqref{con3},  we have
\begin{align}
\begin{split}
\jb{ Z_N (t), \psi}
& = \sum_{n \in \Z^*}
\ft Z_N (t,n) \cj{\ft \psi(n)}\\
& =   C_{\al, N}^2
\sum_{n \in \Z^*}
 e^{t \vp(n)} \cj{\ft \psi(n)}\\
& \quad \times  \sum_{\substack{n= n_1 + n_2 \\ |n_1|, |n_2|\leq N}}
 \int^t_0
 e^{- t' (\vp(n) -\vp(n_1) -\vp(n_2))} 
 \vp(n)
\frac{g_{n_1} g_{n_2}}{\jb{n_1}^\al \jb{n_2}^{\al}}
 dt'.
 \end{split}
\label{con3a}
\end{align}

\noi
It follows from \eqref{mom1} that for $n  \ne 0$,
 we have 
$\E[g_{n_1} g_{n_2}]  = 0$
for $n_1, n_2 \in \Z$, satisfying $n = n_1 + n_2$.
Hence, 
 \eqref{con1} for $Z_N$ follows from \eqref{con3a}.
Note that 
 \eqref{con1}  for $Z$ follows easily from \eqref{lin2}
with \eqref{I3b}.

From \eqref{con3a} with $\cj{\vp(n)} = - \vp(n)$, we have 
\begin{align}
\begin{aligned}
\E& \Big[\jb{Z_N (r), \psi_1} \jb{ Z_N (t), \psi_2}\Big]
= \E \Big[\jb{Z_N (r), \psi_1}\cj{ \jb{ Z_N (t), \psi_2}}\Big]\\
& = \sum_{n, m\in \Z^*}
   \E    \Big[\ft Z_N (r, n)\cj{\ft Z_N (t, m)}\Big]
\cj{\ft \psi_1(n)} \ft \psi_2(m)\\   
&  =  C_{\al, N}^4 
\sum_{n, m\in \Z^*}
\cj{\ft \psi_1(n)} \ft \psi_2(m)
\int^r_0 e^{(r-t_1) \vp(n)} 
\vp(n)
\int^t_0 e^{- (t-t_2)\vp(m)}
\cj{\vp(m)}
\\
& \quad \times 
\sum_{\substack{n= n_1+n_2\\|n_1|, |n_2|\leq N}} 
\sum_{\substack{m= m_1+m_2\\|m_1|, |m_2|\leq N}} 
\frac{e^{t_1(\vp(n_1)+ \vp(n_2))} }{\jb{n_1}^\al \jb{n_2}^{\al}}
\frac{e^{-t_2(\vp(m_1)+ \vp(m_2))} }{\jb{m_1}^\al \jb{m_2}^{\al}}  \\
& \hphantom{XXXXXXXXXXX} \times 
\E\big[g_{n_1} g_{n_2} \cj{g_{m_1} g_{m_2}}\big] dt_1 dt_2.
\end{aligned}
\label{con4}
\end{align}

\noi
Since $n_1 + n_2, m_1 + m_2 \ne 0$, 
we see 
from Wick's theorem \cite[Proposition I.2]{Simon}
with \eqref{mom1}
that 
$\E[g_{n_1} g_{n_2} \cj{g_{m_1} g_{m_2}}]$
is non-zero only if
$(n_1, n_2)  = (m_1, m_2)$ or $(m_2, m_1)$,
which in particular implies
$n = m$.
Moreover,  by separately considering the cases
$n_1 \ne n_2$ and $n_1 = n_2 = \frac n2$ (which occurs
only when $n$ is even) with 
\eqref{mom2}, we have 
\begin{align}
\sum_{\substack{n= m_1+m_2\\|m_1|, |m_2|\leq N}} 
\E\big[g_{n_1} g_{n_2} \cj{g_{m_1} g_{m_2}}\big]
= 2
\label{con5}
\end{align}

\noi
for any  $n_1, n_2 \in \Z$ with $|n_1|, |n_2|\le N$
and $n = n_1 + n_2 \ne 0$.
Thus, from \eqref{con4} and \eqref{con5}, we have 
\begin{align}
\begin{aligned}
\E& \Big[\jb{Z_N (r), \psi_1} \jb{ Z_N (t), \psi_2}\Big]\\
&  =  2C_{\al, N}^4
 \sum_{n \in \Z^*}
\cj{\ft \psi_1(n)} \ft \psi_2(n) 
|\vp(n)|^2
\int^r_0 e^{(r-t_1) \vp(n)} \int^t_0 e^{- (t-t_2)\vp(n)}\\
& \quad \times 
\sum_{\substack{n= n_1+n_2\\|n_1|, |n_2|\leq N}} 
\frac{e^{(t_1- t_2)(\vp(n_1)+ \vp(n_2))} }{\jb{n_1}^{2\al} \jb{n_2}^{2\al}}
 dt_1 dt_2.
\end{aligned}
\label{con6}
\end{align}

In the following, we first consider the case $0 < \al \le \frac 14$.
We then briefly discuss the case $ \al \le 0$.

\smallskip

\noi
{\bf $\bullet$ Case 1:} $0 < \al \le \frac 14$.
\\
\indent
By the Taylor expansion, we have 
\begin{align}
\begin{split}
&  \sum_{\substack{n= n_1+n_2\\|n_1|, |n_2|\leq N}} 
\frac{e^{(t_1- t_2)(\vp(n_1)+ \vp(n_2))} }{\jb{n_1}^{2\al} \jb{n_2}^{2\al}}\\
& \quad = \sum_{\substack{n= n_1+n_2\\|n_1|, |n_2|\leq N}} 
\frac{1 }{\jb{n_1}^{2\al} \jb{n_2}^{2\al}}\\
&  \quad\quad  
+ O\bigg(\sum_{k = 1}^\infty
\frac{|t_1- t_2|^k}{k!}
\sum_{\substack{n= n_1+n_2\\|n_1|, |n_2|\leq N}} 
\frac{ |\vp(n_1)+ \vp(n_2)|^k} {\jb{n_1}^{2\al} \jb{n_2}^{2\al}}\bigg).
\end{split}
\label{con7}
\end{align}

\noi
Let us first consider the summand in the $k$-summation 
on the right-hand side.
From \eqref{phi2}, 
we see that 
 $|\vp(n)|\le \frac 12$
 and 
 that 
$|\vp(n)|$ is decreasing for $n \in \N$.
Then, 
we have 
\begin{align}
\begin{split}
\sum_{\substack{n= n_1+n_2\\|n_1|, |n_2|\leq N}} 
\frac{ |\vp(n_1)+ \vp(n_2)|^k} {\jb{n_1}^{2\al} \jb{n_2}^{2\al}}
& \le \sum_{\substack{n= n_1+n_2\\|n_1|, |n_2|\leq N}} 
\frac{ |\vp(n_1)+ \vp(n_2)|} {\jb{n_1}^{2\al} \jb{n_2}^{2\al}}\\
& \les
\sum_{\substack{n= n_1+n_2\\|n_1|\le |n_2|\leq N}} 
\frac{1} {\jb{n_1}^{2\al+1} \jb{n_2}^{2\al}}\\
& \le \sum_{|n_1|\le N}
\frac{ 1} { \jb{n_1}^{4\al+1}} \sim 1, 
\end{split}
\label{con8}
\end{align}

\noi
uniformly in $k \in \N$ and $N \gg 1$.
Then, from \eqref{con7} and \eqref{con8}, 
we have 
\begin{align}
  \sum_{\substack{n= n_1+n_2\\|n_1|, |n_2|\leq N}} 
\frac{e^{(t_1- t_2)(\vp(n_1)+ \vp(n_2))} }{\jb{n_1}^{2\al} \jb{n_2}^{2\al}}
 = \sum_{\substack{n= n_1+n_2\\|n_1|, |n_2|\leq N}} 
\frac{1 }{\jb{n_1}^{2\al} \jb{n_2}^{2\al}}
+ O(e^{c t})
\label{con9}
\end{align}

\noi
as $N \to \infty$, 
where we used the fact that $|t_1 - t_2| \le t$
for $0 \le t_1 \le r$ and $0 \le t_2 \le t$ with $r\le t$.
Hence, 
from \eqref{con6}, \eqref{con9}, and \eqref{CN1}, 
we obtain
\begin{align}
\begin{aligned}
& \lim_{N \to \infty} \E \Big[\jb{Z_N (r), \psi_1} \jb{ Z_N (t), \psi_2}\Big]\\
& \quad  = c_\al 
 \sum_{n \in \Z^*}
\cj{\ft \psi_1(n)} \ft \psi_2(n) 
|\vp(n)|^2
\int^r_0 e^{(r-t_1) \vp(n)}
 dt_1
\int^t_0 e^{- (t-t_2)\vp(n)}
 dt_2, 
\end{aligned}
\label{con10}
\end{align}

\noi
where $c_\al > 0$ is defined by 
\begin{align}
c_\al 
= \lim_{N\to \infty} 
 C_{\al, N}^4 \sum_{\substack{n=n_1+n_2\\ |n_1|,|n_2|\le N }} 
\frac{2}{\jb{n_1}^{2\al} \jb{n_2}^{2\al}} .
\label{con11}
\end{align}

\noi
While it 
may seem that $c_\al$ depends on $n \in \Z$, 
we point out  that $c_\al$ is independent of $n \in \Z$.
Moreover, we claim that 
\begin{align}
c_\al = 1.
\label{AA5}
\end{align}

\noi
See Remark \ref{REM:CA1} below.

On the other hand, from \eqref{lin2} with \eqref{I3b}, 
$\E[ \gf_n \cj \gf_m] = \ind_{n = m}$, 
and  $\cj{\vp(n)} = - \vp(n)$,  we have 
\begin{align}
\begin{split}
\E& \Big[\jb{ Z (r), \psi_1}\jb{ Z (t), \psi_2}\Big]\\
&   =  
   \sum_{n \in \Z^*}
\cj{\ft \psi_1(n)} \ft \psi_2(n) 
   |\vp(n)|^2
\int^r_0 e^{(r-t_1) \vp(n)}
 dt_1
\int^t_0 e^{- (t-t_2)\vp(n)}
 dt_2.
 \end{split}
\label{con12}
\end{align}

\noi
Therefore, 
by comparing \eqref{con10} and \eqref{con12}, 
we obtain
 \eqref{con2}
when $0 < \al \le \frac 14$.

\begin{remark}\label{REM:CA1}
\rm

Let  $0 < \al \le \frac 14$.
The definition \eqref{con11}
 of  $c_\al  > 0 $ a priori depends on $n\in \Z$
but we show that $c_\al$ is in fact independent of $n \in \Z$.

We first consider the case
  $0 < \al <  \frac 14$.
Note that the contribution 
to the summation in~\eqref{con11} 
from $\min(|n_1|, |n_2|) \les N^{1-\eps}$
vanishes in the limit. 
Indeed, from~\eqref{CN}, we have 
\begin{align}
\begin{split}
& \lim_{N\to \infty} 
C_{\al, N}^4 
\sum_{\substack{n=n_1+n_2\\ |n_1|,|n_2|\le N\\ \min(|n_1|, |n_2|) \les N^{1-\eps} }} 
\frac{2}{\jb{n_1}^{2\al} \jb{n_2}^{2\al}} \\
& \quad \les 
\lim_{N\to \infty} 
C_{\al, N}^4 
\sum_{|k| \les N^{1-\eps}} \frac 1{\jb{k}^{4\al}}
\too 0,  
\end{split}
\label{AA1}
\end{align}

\noi
as $N \to \infty$.
Next, we consider the contribution 
from $\min(|n_1|, |n_2|) \gg N^{1-\eps}$.
Since $n \in \Z$ is fixed, 
given any small $\ta > 0$, there exists $N_0 = N_0(n, \al, \ta)\in \N$ such that 
\begin{align}
1 - \ta \le \frac{\jb{n_2}^{2\al}}{\jb{n_1}^{2\al}}
= \frac{\jb{n - n_1 }^{2\al}}{\jb{n_1}^{2\al}} \le 1 + \ta 
\label{AA2}
\end{align}

\noi
for any $N \ge N_0$
and $n_1, n_2 \in \Z$ with  $\min(|n_1|, |n_2|) \gg N^{1-\eps}$.
Then, from \eqref{con11} with \eqref{AA1}, \eqref{AA2}, 
and \eqref{CN}, we obtain
\begin{align}
\begin{split}
c_\al 
& = \lim_{N\to \infty} 
C_{\al, N}^4 \sum_{\substack{n=n_1+n_2 \\ N^{1-\eps} \ll |n_1|,|n_2|\le N }} 
\frac{2}{\jb{n_1}^{2\al} \jb{n_2}^{2\al}} \\
& \le \frac 1{1-\ta}
\lim_{N\to \infty} 
C_{\al, N}^4 \sum_{|n_1|\le N }
\frac{2}{\jb{n_1}^{4\al}}
\le \frac 1{1-\ta}
\end{split}
\label{AA3}
\end{align}

\noi
and, similarly, 
\begin{align}
c_\al 
& \ge \frac 1{1+\ta}
\lim_{N\to \infty} 
C_{\al, N}^4 \sum_{|n_1|\le N }
\frac{2}{\jb{n_1}^{4\al}}
\ge \frac 1{1+\ta}.
\label{AA4}
\end{align}

\noi
Since the choice of small $\ta > 0$ was arbitrary, 
the identity \eqref{AA5} follows from \eqref{AA3} and~\eqref{AA4}.

When $\al = \frac 14$, 
we
note  that the contribution 
to the summation in~\eqref{con11} 
from $\min(|n_1|, |n_2|) \les \log N$
vanishes in the limit. 
Then, by arguing as in  \eqref{AA3} and \eqref{AA4}, 
we obtain \eqref{AA5} in this case as well.

\end{remark}

\smallskip

\noi
{\bf $\bullet$ Case 2:} $ \al \le 0$.
\\
\indent
We first estimate the contribution 
from $k = 1$ in the second term on the right-hand side of~\eqref{con7}.
In this case, we consider 
\begin{align}
C_{\al, N}^4 \sum_{\substack{n= n_1+n_2\\|n_1|, |n_2|\leq N}} 
\frac{ |\vp(n_1)+ \vp(n_2)|} {\jb{n_1}^{2\al} \jb{n_2}^{2\al}}
\label{AB1}
\end{align}

\noi
instead of \eqref{con8}.
Without loss of generality, assume $|n_1| \le |n_2|$.
The contribution to \eqref{AB1}
from  $ |n_2| \les N^{1-\eps}$
is bounded by 
\begin{align}
\begin{split}
& \le C_{\al, N}^4 \sum_{\substack{n= n_1+n_2\\|n_1| \le  |n_2|\les N^{1-\eps}}} 
\frac{1} {\jb{n_1}^{2\al} \jb{n_2}^{2\al}}
\les 
C_{\al, N}^4 \sum_{|n_2| \les N^{1-\eps}}
\jb{n_2}^{-4\al}\\
& \too 0, 
\end{split}
\label{AB2}
\end{align}

\noi
as $N \to \infty$, where the last step follows from 
\eqref{CN}.
By Cauchy-Schwarz's inequality with~\eqref{CN},  the contribution to \eqref{AB1}
from  $ |n_1| \les |n_2|^{1-\eps}$
is bounded by 
\begin{align}
\begin{split}
& \le C_{\al, N}^4 \sum_{\substack{n= n_1+n_2\\|n_1| \les |n_2|^{1-\eps}
\le N^{1-\eps}
}} 
\frac{1} {\jb{n_1}^{2\al} \jb{n_2}^{2\al}}\\
& \les 
C_{\al, N}^4 
\bigg(
\sum_{|n_1| \les N^{1-\eps}}
\jb{n_1}^{-4\al}\bigg)^\frac 12
\bigg(
\sum_{|n_2| \les N}
\jb{n_2}^{-4\al}\bigg)^\frac 12\\
& \too 0, 
\end{split}
\label{AB3}
\end{align}

\noi
as $N \to \infty$.
It remains to consider the case
$|n_1| \gg |n_2|^{1-\eps}$
and 
$|n_2| \gg N^{1-\eps}$, 
which in particular implies
$|n_2| \ge |n_1| \gg N^{(1-\eps)^2}$.
Recall that $|\vp(n)|$ is decreasing for $n \in\N$.
Then, using~\eqref{phi2} and \eqref{CN}, 
we can bound 
the contribution to \eqref{AB1}
in this case by 
\begin{align}
\begin{split}
& \le C_{\al, N}^4 
N^{-(1-\eps)^2}
\sum_{\substack{n= n_1+n_2\\|n_1| \le  |n_2|\les N}} 
\frac{1} {\jb{n_1}^{2\al} \jb{n_2}^{2\al}}
\les 
N^{-(1-\eps)^2}\\
& \too 0, 
\end{split}
\label{AB4}
\end{align}

\noi
as $N \to \infty$. 

Hence, from \eqref{con7}, 
\eqref{AB2}, \eqref{AB3}, and \eqref{AB4}, 
we obtain
\begin{align*}
 C_{\al, N}^4   \sum_{\substack{n= n_1+n_2\\|n_1|, |n_2|\leq N}} 
\frac{e^{(t_1- t_2)(\vp(n_1)+ \vp(n_2))} }{\jb{n_1}^{2\al} \jb{n_2}^{2\al}}
 =  C_{\al, N}^4 \sum_{\substack{n= n_1+n_2\\|n_1|, |n_2|\leq N}} 
\frac{1 }{\jb{n_1}^{2\al} \jb{n_2}^{2\al}}
+ o(1), 
\end{align*}

\noi
as $N \to \infty$.
As a consequence, \eqref{con10} holds with $c_\al$ as in \eqref{con11}.

We note that \eqref{AA5} still holds in this case.
Indeed, from \eqref{AB2}
and \eqref{AB3}, the contribution to \eqref{con11}
vanishes unless
$|n_1| \gg |n_2|^{1-\eps}$
and 
$|n_2| \gg N^{1-\eps}$, 
which in particular implies
$|n_1| \gg N^{(1-\eps)^2}$.
Then, we can proceed as in 
Remark \ref{REM:CA1} with \eqref{AA2}
to deduce \eqref{AA5}.
Hence, by comparing 
\eqref{con10} and \eqref{con12} with \eqref{AA5}, 
we  obtain
  \eqref{con2}
  in this case.

This concludes the proof of Lemma \ref{LEM:COV}.
\end{proof}

We now turn to a proof of Proposition \ref{PROP:uniq}.
In view of Lemma \ref{LEM:BDW}, 
Proposition \ref{PROP:uniq}
follows once we  prove the following lemma.

\begin{lemma}\label{LEM:uniq2}
Let $\al\le \frac 14$.
Let $Z_N$ and $Z$ be as in \eqref{I3} and \eqref{lin2}, respectively.
Given $m \in \N$, let 
 $0\le t_1\le  \cdots \le t_m$.
Then, given any test functions $\{\psi_j\}_{j = 1}^m \subset \D(\T)$, 
$\big\{\jb{Z_N(t_j), \psi_j}\big\}_{j = 1}^m$
converges in law
to  the Gaussian vector 
$\big\{\jb{Z(t_j), \psi_j}\big\}_{j = 1}^m$
as $N \to \infty$.

\end{lemma}

\begin{proof}

%

Fix  $t \ge 0$
and a non-zero test function $\psi \in \D(\T)$.
It follows from 
\eqref{con3a} with $\vp(n)|_{n = 0} = 0$
that $\jb{Z_N(t), \psi}$ belongs to 
the second homogeneous Wiener chaos $\H_2$. 
Then, from \eqref{con3a} and
 \eqref{WW2} with~\eqref{WW1}, 
we write 
 $\jb{Z_N(t), \psi}$ as
\begin{align}
\jb{Z_N(t), \psi}
= I_2(\fn^t),
\label{B1}
\end{align}

\noi
where $\fn^t = \fn^t(x_1, x_2) \in L^2(\T^2)$ is a symmetric function
whose Fourier transform is given by 
\begin{align}
\ft {\fn^t}(n_1, n_2)
&
 = 
 \ind_{ |n_1|, |n_2|\leq N}
 \cdot 
 C_{\al, N}^2
\frac{ e^{t \vp(n_1 + n_2)} \vp(n_1+n_2)\cj{\ft \psi(n_1+n_2)}}{\jb{n_1}^\al \jb{n_2}^{\al}}
J_{n_1, n_2}(t)
\label{B2}
\end{align}

\noi
for $n_1, n_2\in \Z$.
Here,  $J_{n_1, n_2}(t)$ is defined by 
\begin{align*}
J_{n_1, n_2}(t)
=  \int^t_0
 e^{- t' (\vp(n_1+n_2) -\vp(n_1) -\vp(n_2))}  dt'.
\end{align*}

Our first goal is to prove the following limit:
\begin{align}
\lim_{N\to\infty} \|  \fn^t \otimes_1 \fn^t \|_{L^2(\T^2) } =0.
\label{B2b}
\end{align}

\noi
By letting
\begin{align}
Y_{n_j, m} (t)
= \frac{e^{t \vp(n_j +m)} \vp(n_j+m)\cj{\ft \psi(n_j+m)}  }{\jb{n_j}^{\al} \jb{m}^{\al} }
J_{n_j, m}(t) , 
\label{B2z}
\end{align}

\noi
it follows from 
\eqref{N3a} and \eqref{B2} that 
\begin{align*}
& \fn^t \otimes_1 \fn^t (x_1, x_2) 
= \int_{\T} \fn^t (x_1, y) \fn^t (x_2, y) dy\\
& \quad = \sum_{m \in \Z}\F_2(\fn^t) (x_1, m)
 \cj{\F_2 (\fn^t) (x_2, m)}\\
& \quad =  C_{\al, N}^4
\sum_{|n_1|, |n_2|, |m|\le N}
e^{i(n_1 x_1 +n_2x_2) } 
Y_{n_1, m}(t)\cj{Y_{n_2, m}(t)}, 
\end{align*}

\noi
where $\F_2(f)$ denotes the Fourier transform 
of $f$ on $\T^2$ only in the second variable.
Then, by applying Plancherel's identity
with \eqref{B2z}, \eqref{phi2} (in particular, $\vp(n)$ is purely imaginary and $|\vp(n)|\le \frac 12$),  and  $|J_{n, n_1}(t)|\le t$, 
 we have
\begin{align}
\begin{split}
& \| \fn^t \otimes_1 \fn^t \|_{L^2(\T^2) }^2\\
& \quad \le  t^4 C_{\al, N}^8
\sum_{|n_1|, |n_2|\le N}
\bigg|\sum_{|m|\le N}
\frac{|\ft \psi(n_1+m)||\ft \psi(n_2+m)|}{\jb{n_1}^{\al}
\jb{n_2}^{\al}
\jb{m}^{2\al}}\bigg|^2.
\end{split}
\label{B4}
\end{align}

\noi
In the following, we bound the right-hand side
by dividing the argument into three cases.
As in Remark \ref{REM:CA1}, we first discuss  the case $0 < \al < \frac 14$
in detail.

\medskip

\noi
{\bf $\bullet$ Case 1:} $\max(|n_1+ m|, |n_2+ m|) \ges N^\ta$
for some small $\ta = \ta(\al)>0$ (to be chosen later).
\\ \indent
Let $A_1$ denote the contribution to the right-hand side of \eqref{B4} from this case.
Without loss of generality, assume
$|n_1+ m| \ges N^\ta$.
Then, by the fast decay of $\ft \psi$, we have
\begin{align}
|\ft \psi(n_1 + m)| \les \jb{n_1 + m}^{-100\ta^{-1}} \les N^{-100}.
\label{B5}
\end{align}

\noi
By applying \eqref{B5} and Cauchy-Schwarz's inequality (in $m$)
 with \eqref{CN}, 
we have 
\begin{align}
\begin{split}
\lim_{N\to \infty}A_1 
& \les 
t^4 
\|\ft \psi\|_{\l^\infty_n} 
\lim_{N\to \infty} N^{-100} C_{\al, N}^8
\bigg( \sum_{|n|,  |m|\le N}
\frac{|\ft \psi(n+m)|}{\jb{n}^{2\al}
\jb{m}^{2\al}}\bigg)^2\\
& \le 
t^4 
\|\ft \psi\|_{\l^\infty_n} 
\|\ft \psi\|_{\l^1_n}^2 
\lim_{N\to \infty}
N^{-100} C_{\al, N}^8 
\bigg( \sum_{|n|\le N}
\frac{1}{\jb{n}^{4\al}}\bigg)^2\\
& = 0, 
\end{split}
\label{B6}
\end{align}

\noi
where the second step follows from 
 Cauchy-Schwarz's inequality (in $n$)
and Young's inequality.

\medskip

\noi
{\bf $\bullet$ Case 2:} $\max(|n_1+ m|, |n_2+ m|) \ll N^\ta$
and $\min(|n_1|, |n_2|, |m|) \les N^\ta$.
\\ \indent
Let $A_2$ denote the contribution to the right-hand side of \eqref{B4} from this case.
In this case, we must have 
$\max(|n_1|, |n_2|, |m|) \les N^\ta$.
Thus, from \eqref{B4} with \eqref{CN}, we have 
\begin{align}
\begin{split}
\lim_{N\to \infty}A_2
& \les 
t^4 \|\ft \psi\|_{\l^\infty_n}^4 
\lim_{N\to \infty}  C_{\al, N}^8
\sum_{|n_1|, |n_2|\les N^\ta}
\bigg|\sum_{|m|\les N^\ta} 1\bigg|^2\\
& \les 
t^4 \|\ft \psi\|_{\l^\infty_n}^4 
\lim_{N\to \infty}  N^{4\ta}C_{\al, N}^8\\
& = 0, 
\end{split}
\label{B7}
\end{align}

\noi
provided that $\ta = \ta(\al) > 0$ is sufficiently small such that
$\ta < \frac 12 - 2\al$.

\medskip

\noi
{\bf $\bullet$ Case 3:} $\max(|n_1+ m|, |n_2+ m|) \ll N^\ta$
and $\min(|n_1|, |n_2|, |m|) \gg N^\ta$.
\\ \indent
Let $A_3$ denote the contribution to the right-hand side of \eqref{B4} from this case.
In this case, we must have $|n_1|\sim|n_2|\sim |m|$.
Furthermore, for fixed $n_1 \in \Z$ with $N^\ta \ll |n_1|\le N$, 
there are $O(N^\ta)$-many choices for each of $n_2$ and $m$, 
satisfying the constraints.
Then, by summing over $n_2$ and $m$ for fixed $n_1$
and applying \eqref{CN}, we have
\begin{align}
\begin{split}
\lim_{N \to \infty}  A_3
 \les  t^4 
 \|\ft \psi\|_{\l^\infty_n}^4 
 \lim_{N \to \infty} 
 N^{3\ta} C_{\al, N}^8
\sum_{|n_1|\le N}\frac 1{\jb{n_1}^{8\al}}
= 0, 
\end{split}
\label{B8}
\end{align}

\noi
provided that 
$0 < \ta <  \min \big(\frac {2}{3} - \frac {8}{3} \al, \frac {1}{3}\big)$.

\medskip
Therefore, we obtain \eqref{B2b} from \eqref{B4}, 
\eqref{B6}, \eqref{B7}, and \eqref{B8}, 
when $0 < \al < \frac 14$.
The required modification for the $\al = \frac 14$ case is straightforward
(as in Remark \ref{REM:CA1})
and thus we omit details.
When $\al \le 0$, 
Cases 1 and 3 hold as they are written.
As for Case 2, we replace~\eqref{B7}
by 
\begin{align*}
\lim_{N\to \infty}A_2
& \les 
t^4 \|\ft \psi\|_{\l^\infty_n}^4 
\lim_{N\to \infty} N^{-8 \al\ta}  C_{\al, N}^8
\sum_{|n_1|, |n_2|\les N^\ta}
\bigg|\sum_{|m|\les N^\ta} 1\bigg|^2\\
& \les 
t^4 \|\ft \psi\|_{\l^\infty_n}^4 
\lim_{N\to \infty}  N^{(4 - 8\al) \ta}C_{\al, N}^8\\
& = 0, 
\end{align*}

\noi
provided that $\ta = \ta(\al) > 0$ is sufficiently small such that
$\ta < \frac {1- 4 \al}{2 - 4\al}$.
Therefore, \eqref{B2b} holds
for any $\al \le \frac 14$.

\medskip

We now prove the main claim.
Without loss of generality, assume $\psi_j\not\equiv0$ for any $j = 1, \dots, m$.
In view of \eqref{B1}, 
we will apply the fourth moment theorem (Lemma \ref{LEM:FMT})
to the sequence $\{F_N\}_{N \in \N}$ given by 
\begin{align*}
 F_N & = \big(\jb{Z_N(t_1), \psi_1}, \dots, \jb{Z_N(t_m), \psi_m}\big)\\
 & = \big(I_2(f^{t_1}_{N, \psi_1}),\dots I_2(f^{t_m}_{N, \psi_m})\big).
\end{align*}

\noi
It follows from Lemma \ref{LEM:COV}
that 
an $m\times m$  matrix  $\Cs=(\Cs_{ij})_{1\le i,j\le m}$ 
defined by 
\begin{align*}
\Cs_{ij} = \lim_{N\to \infty} 
\E\Big[\jb{Z_N (t_i), \psi_i} \jb{ Z_N (t_j), \psi_j}\Big]
= 
\E\Big[\jb{ Z (t_i), \psi_i}\jb{ Z (t_j), \psi_j}\Big]
\end{align*}

\noi
is positive semi-definite.\footnote{Recall that every covariance matrix is positive semi-definite;
see \cite[pp.86-87]{Kal}.}
From this observation
 and~\eqref{B2b}, 
we then invoke
the fourth moment theorem (Lemma~\ref{LEM:FMT})
with~\eqref{B1}
and 
 conclude that
$\big\{\jb{ Z_N(t_j), \psi_j}\big\}_{j = 1}^m$
converges in law
to 
$\big\{\jb{ Z(t_j), \psi_j}\big\}_{j = 1}^m$
as $N \to \infty$.
This proves Lemma \ref{LEM:uniq2}, 
and hence, Proposition \ref{PROP:uniq} in view of Lemma \ref{LEM:BDW}.
\end{proof}

\subsection{Tightness}
\label{SUBSEC:Z2}

Next, we show  tightness for 
$\{Z_N\}_{N \in \N}$.

\begin{proposition}\label{PROP:tight}
Let $\al\le \frac 14$
and $Z_N$ be as in \eqref{I3}.
Then, given any  $s < \frac 12$, 
the sequence $\{Z_N\}_{N \in \N}$ is tight in 
 $C(\R_+;W^{s,\infty}(\T))$.

\end{proposition}

We first present a proof of Theorem \ref{THM:1}, 
assuming 
Proposition \ref{PROP:tight}.

\begin{proof}[Proof of Theorem \ref{THM:1}]
Let $s < \frac 12$. Given a subsequence $\{Z_{N_j}\}_{j \in \N}$
of $\{Z_N\}_{N \in \N}$, 
it follows from  the Prokhorov theorem (Lemma \ref{LEM:Pro})
that 
there exists a subsubsequence
$\{Z_{N_{j_k}}\}_{k \in \N}$,
converging in law to some limit $Z_\infty$ 
in $C(\R_+;W^{s,\infty}(\T))$
as $k \to \infty$.
On the other hand, 
from Proposition~\ref{PROP:uniq}, 
we see that such a limit is unique
as $\D'(\T)$-valued process in time, 
which in particular guarantees uniqueness in law of the limit 
in  $C(\R_+;W^{s,\infty}(\T))$.
Namely, $Z_\infty = Z$ in law, 
where $Z$ is as in \eqref{lin2}.
Therefore, we conclude that the entire sequence $\{Z_N\}_{N\in \N}$
converges in law to the unique limit $Z$
 in $C(\R_+;W^{s,\infty}(\T))$ as $N \to \infty$.
This proves Theorem~\ref{THM:1}.
\end{proof}

We conclude this section by presenting a proof of Proposition \ref{PROP:tight}.
While it is elementary and well known,\footnote{For example, 
the second half of the proof of Proposition \ref{PROP:tight}
immediately follows from \eqref{dl2}, \eqref{dl5}, and 
 \cite[Theorem 23.7]{Kal}.} 
we present some details for readers' convenience.

\begin{proof}[Proof of Proposition \ref{PROP:tight}]

Let $t \in \R_+$.
From \eqref{con6} with $\psi = e_n$
and Cauchy-Schwarz's inequality with \eqref{CN}
and \eqref{phi2}, we have 
\begin{align}
\begin{aligned}
\E \Big[|\ft Z_N (t, n)|^2\Big]
&  =  2C_{\al, N}^4
|\vp(n)|^2
\int^t_0 e^{-t_1 \vp(n)} \int^t_0 e^{t_2\vp(n)}\\
& \quad \times 
\sum_{\substack{n= n_1+n_2\\|n_1|, |n_2|\leq N}} 
\frac{e^{(t_1- t_2)(\vp(n_1)+ \vp(n_2))} }{\jb{n_1}^{2\al} \jb{n_2}^{2\al}}
 dt_1 dt_2\\
 & \les t^2 
 C_{\al, N}^4
 |\vp(n)|^2
\sum_{|n_1|\le N}\frac{1}{\jb{n_1}^{4\al}}\\
&  \les t^2  \jb{n}^{-2}.
\end{aligned}
\label{TT1}
\end{align}

\noi
Hence, from  Lemma \ref{LEM:OOT}\,(i), 
we have $Z_N(t)\in W^{s,\infty}(\T)$ for any $s<\frac 12$,
almost surely.
Moreover, from the mean value theorem with \eqref{phi2}, we have
\begin{align}
|e^{(t+h)\vp(n)} - e^{t\vp(n)}| \le |h|
\label{TT2}
\end{align}

\noi
for each $t \in \R_+$ and $h \in \R$.
A slight modification of \eqref{TT1} with \eqref{TT2}
yields
\begin{align*}
\E \Big[|\dl_h \ft Z_N (t, n)|^2\Big]
\les_t  \jb{n}^{-2}|h|^2
\end{align*}

\noi
for any $h \in \R$, 
where $\dl_h$ is the difference operator defined in \eqref{dl1}.
Since the argument is standard, we omit details;
see, for example, the proof of \cite[Lemma 3.1\,(i)]{GKO2}.
Hence, from 
 Lemma \ref{LEM:OOT}\,(ii), 
 we conclude that 
 $Z_N\in C(\R_+; W^{s,\infty}(\T))$ for any $s<\frac 12$,
almost surely.

Let $s < s_1 < \frac 12$ and $0 < \be < \frac 12$.
Given $\dl > 0$, define $K_\dl$ by 
\begin{align}
 K_\dl = \big\{ u \in C(\R_+; W^{s, \infty}(\T)):\, \| u \|_{\C^\be_{T_j} W_x^{s_1, \infty}} \leq  c_0 \dl^{-\frac 12} T_j 
\text{ for all }
j \in \N \big\}, 
\label{TT4}
\end{align}

\noi
for some $c_0 \gg1 $, 
where  $T_j = 2^j$.
Then, by applying 
Lemma \ref{LEM:OOT}\,(iii) (with $p = 2$)
and choosing $c_0 \gg 1$ sufficiently large, we have 
\[ \PP(K_\dl^c) 
\leq C_{\be, 2} c^{-2}_0  \dl \sum_{j = 1}^\infty T_j^{-1}
=  C_{\be, 2}c^{-2}_0 \dl < \dl.\]

Hence, it remains to prove that $K_\dl$ is compact in $C(\R_+; W^{s, \infty}(\T))$
endowed with the compact-open topology in time.
Recall from Rellich's lemma and the Arzel\`a-Ascoli theorem
\cite[Theorem~40 in Section~7.10]{Royden}
that  the embedding 
\[\C^\be([0, T];  W^{s_1, \infty}(\T)) \subset 
C([0, T]; W^{s, \infty}(\T))\]

\noi
 is compact
for each $T>0$.
Let $\{u_n \}_{n \in \N} \subset K_\dl$.
By the definition \eqref{TT4} of $K_\dl$, 
$\{u_n \}_{n \in \N}$ is bounded in 
$\C^\be([0, T_j]; W^{s_1, \infty}(\T))$ for each $j \in \N$.
Then, by a diagonal argument, 
we can extract a subsequence $\{u_{n_k} \}_{k \in \N}$
convergent in 
$C([0, T_j]; W^{s, \infty}(\T))$ for each $j \in \N$.
In particular,  $\{u_{n_\l} \}_{\l \in \N}$ converges uniformly in $W^{s, \infty}(\T)$
on any compact time interval.
Hence,  $\{u_{n_k} \}_{k \in \N}$ converges in $C(\R_+; W^{s, \infty}(\T))$
endowed with the compact-open topology.
This proves that $K_\dl$ is compact in $C(\R_+; W^{s, \infty}(\T))$, 
thus concluding the proof of Proposition~\ref{PROP:tight}.
\end{proof}

\begin{remark}\label{REM:indep}\rm

Let $\al \le \frac 14$. 
Given $N \in \N$, 
let  $u_0$, $Z_N$,  and $Z$ be  as in \eqref{data1}, \eqref{I3} and~\eqref{lin2}, respectively.
Then, from Lemma \ref{LEM:COV}, 
we have 
\begin{align*}
& \lim_{N\to \infty} 
\E\Big[\jb{(\P_N u_0, Z_N (r)), \psi_1}_2 \jb{ (\P_N u_0, Z_N (t)), \psi_2}_2\Big]\\
& \quad = 
\E\Big[\jb{(u_0,  Z (r)), \psi_1}_2\jb{ (u_0, Z (t)), \psi_2}_2\Big]
\end{align*}

\noi
\noi
for any test functions $\psi_1, \psi_2 \in (\D(\T))^{\otimes 2}$ and $t \ge r \ge 0$, 
where 
$\jb{f, g}_2$ is as in \eqref{dual1}
with $m = 2$.
Moreover, 
a slight modification of the proof of Lemma~\ref{LEM:uniq2},
using the fourth moment theorem (Lemma~\ref{LEM:FMT}),  shows that,
given any test functions $\{\psi_j\}_{j = 0}^m \subset \D(\T)$,  
the vector 
\[ F_N = \big( \jb{\P_N u_0, \psi_0},  \jb{Z_N(t_1), \psi_1}, \dots, 
\jb{Z_N(t_m), \psi_m}\big)\]

\noi
converges in law
to the jointly Gaussian vector
\[ F = \big( \jb{u_0, \psi_0}, 
\jb{Z(t_1), \psi_1}, \dots, 
\jb{Z(t_m), \psi_m}\big),\]

\noi
as $N \to \infty$.
Noting from 
Lemma \ref{LEM:OOT}
and 
Proposition \ref{PROP:tight}
that 
the sequence  $\big\{(\P_N u_0, Z_N)\big\}_{N \in \N}$
is tight in 
 $W^{\al - \frac 12 - \eps, \infty}(\T)
 \times C(\R_+; W^{\frac 12 - \eps, \infty}(\T))$, 
we conclude that  
 $(\P_N u_0, Z_N)$
 converges in law
 to the jointly Gaussian process
 $(u_0, Z)$
 in
 $W^{\al - \frac 12 - \eps, \infty}(\T)
 \times C(\R_+; W^{\frac 12 - \eps, \infty}(\T))$
 as $N \to \infty$.

Given  any test functions $\psi_1, \psi_2 \in \D(\T)$, $t \in \R_+$, 
and $N \in \N$, 
it follows from 
\eqref{data1}, \eqref{con3a}
and 
Wick's theorem \cite[Proposition I.2]{Simon}
that\footnote{Alternatively, 
by noting that 
$\jb{u_0, \psi_1}  \in \H_1$ and $\jb{ Z_N (t), \psi_2} \in \H_2$, 
\eqref{Zcon2} follows from the orthogonality 
of $\H_1$ and $\H_2$.}
\begin{align}
\E\Big[\jb{u_0, \psi_1} \jb{ Z_N (t), \psi_2}\Big] = 0.
\label{Zcon2}
\end{align}

\noi
Hence, by taking a limit as $N \to \infty$, 
it follows from \eqref{Zcon2}
and Theorem~\ref{THM:1} that 
\begin{align*}
\E\Big[\jb{u_0, \psi_1}\jb{ Z (t), \psi_2}\Big]
= 
\lim_{N\to \infty} 
\E\Big[\jb{u_0, \psi_1} \jb{ Z_N (t), \psi_2}\Big] = 0,
\end{align*}

\noi
which implies $u_0$ in \eqref{data1}
and the Gaussian process $Z$ 
constructed in Theorem~\ref{THM:1}
are independent;
see, for example, \cite[Theorem 2.4.3]{Stroock}.
Hence, recalling  that $Z(t) = - \I(\vp(D) \ze)$, 
we conclude that $u_0$ and $\ze$ are independent.

\end{remark}

\section{On the limiting equation}
\label{SEC:limit}

In this section, we 
prove almost sure global well-posedness
of the limiting equation \eqref{BBM4d}
(Proposition \ref{PROP:2}).
In view of the almost sure regularity property 
of $Z$ established in Theorem~\ref{THM:1}, 
we see that 
Proposition \ref{PROP:2}
follows once we prove the following
global well-posedness of the perturbed BBM \eqref{XM1}.

\begin{proposition}\label{PROP:WP1}
Let $f \in C(\R_+; W^{\frac 12 - \eps,\infty}(\T))$.
Then, there exists a global-in-time solution 
$v \in C(\R_+; H^1(\T))$
to the following perturbed BBM\textup{:}
\begin{align}
\begin{cases}
\dt   v=  \vp(D) v +  \vp(D) 
\big( (v +  f)^2 \big) \\
 v|_{t=0} =0.
 \end{cases}
\label{XM1}
\end{align}

\end{proposition}

\subsection{Local well-posedness}
\label{SUBSEC:LWP}

In this subsection, 
we prove local well-posedness of \eqref{XM1}
with general initial data $v_0$:
\begin{align}
\begin{cases}
\dt   v=  \vp(D) v +  \vp(D) 
\big( (v +  f)^2 \big) \\
 v|_{t=t_0} =v_0.
 \end{cases}
\label{XM2}
\end{align}

\begin{lemma}\label{LEM:LWP}
Given $t_0 \ge 0$, 
let $f \in C([t_0, t_0 + 1]; W^{s,\infty}(\T))$
for some
$s \ge 0$.
Let $0 \le \s \le 1$.
Then, given any $v_0 \in H^\s(\T)$, 
there exists a unique solution $v
\in  C([t_0 , t_0 + \tau]; H^\s(\T))$ to~\eqref{XM2}, 
where the local existence time $\tau \in (0, 1]$ 
depends only on 
$\|v_0\|_{H^\s}$ and $\|f \|_{C([t_0, t_0+ 1]; W^{s, \infty}_x)}$.
\end{lemma}

\begin{proof}
Without loss of generality, we assume that $t_0 = 0$.
Then, by writing \eqref{XM2} in the Duhamel formulation, we have 
\begin{align}
v(t)  = S(t) v_0 + \I\big(\vp(D)(v^2)\big)(t)
+ 2 \I\big(\vp(D)(v f)\big)(t)
+ \I\big(\vp(D)(f^2)\big)(t), 
\label{XM3}
\end{align}

\noi
where $\I$ is the Duhamel integral operator defined in \eqref{exp2}.
From Lemma \ref{LEM:BT}, we have
\begin{align}
\| \vp(D)(v^2)\|_{C_T H^\s_x}
& \les \|v \|_{C_T H^\s_x}^2
\label{XM4}
\end{align}

\noi
for $\s \ge 0$.
From the smoothing property of $\vp(D)$ in \eqref{phi1}
and H\"older's inequality, 
we have 
\begin{align}
\begin{split}
\| \vp(D)(vf)\|_{C_T H^\s_x}
& \les 
\| vf\|_{C_T L^2_x}
 \le \|v \|_{C_T L^2_x}
\| f\|_{C_T L^{\infty}_x}\\
& \le \|v \|_{C_T H^\s_x}
\| f\|_{C_T W^{s, \infty}_x}
\end{split}
\label{XM5}
\end{align}

\noi
for $0 \le \s \le 1$ and $ s \ge 0$.
Similarly, we have 
\begin{align}
\| \vp(D)(f^2)\|_{C_T H^\s_x}
& \les 
\| f^2\|_{C_T L^2_x}
\le \| f\|_{C_T W^{s, \infty}_x}^2
\label{XM6}
\end{align}

\noi
\noi
for $\s \le 1$ and $ s \ge 0$.
Then, the claimed local well-posedness
of \eqref{XM2} follows from a standard contraction argument
applied to \eqref{XM3} with \eqref{XM4}, 
\eqref{XM5}, and \eqref{XM6}.
We omit details.
\end{proof}

\subsection{Global well-posedness}
\label{SUBSEC:GWP}

In view of Lemma \ref{LEM:LWP}, 
Proposition \ref{PROP:WP1} follows
once we prove
that the $H^1$-norm of the solution $v$ to \eqref{XM1}
remains finite on each finite time interval.
Namely, 
\begin{align}
\sup_{0 \le t \le T} \| v(t)\|_{H^1} \le C\big(T, 
\|f\|_{C_T W^{\frac 12 - \eps, \infty}_x}\big) < \infty
\label{XM7}
\end{align}

\noi
for each $T \gg 1$.

Fix $T \gg 1$.
By multiplying \eqref{XM1}
by $(1-\dx^2) v$ and integrating by parts, we have 
\begin{align*}
\frac{d}{dt} \| v(t) \|_{H^1_x}^2 
=  \int_\T  f \dx v^2  dx
+ \int_\T  f^2 \dx v  dx, 
\end{align*}

\noi
where, for simplicity of the presentation, 
we suppressed the $t$-dependence on the right-hand side.
By H\"older's inequality,  the algebra property of $H^1(\T)$, 
and Cauchy's inequality,  
we have 
\begin{align}
\begin{split}
\frac{d}{dt} \| v(t) \|_{H^1_x}^2
& \le 
\|f\|_{C_T L^2_x}\|v^2(t)\|_{H^1_x}
+ \|f\|_{C_T L^4_x}^2\|v(t)\|_{H^1_x}\\
& \les 
\|f\|_{C_T L^4_x}^4
+ \big(1 + 
\|f\|_{C_T L^2_x}\big)\|v(t)\|_{H^1_x}^2.
\end{split}
\label{XM8}
\end{align}

\noi
Then, by applying Gronwall's inequality with \eqref{XM8}, 
we obtain \eqref{XM7}.
This proves Proposition~\ref{PROP:WP1}
and hence
Proposition \ref{PROP:2}.

\section{Proof of Theorem \protect\ref{THM:3}}
\label{SEC:conv}

In this section, we present a proof of Theorem \ref{THM:3}.
The following lemma summarizes the 
convergence result on $(z_N, Z_N)$ needed to prove 
Theorem \ref{THM:3}.

\begin{lemma}
\label{LEM:sto}

Given $\al\leq \frac14$, 
let $z_N$,   $Z_N$, and $Z$  be as in \eqref{lin1}, \eqref{I3}, and \eqref{lin2}, respectively.

\medskip

\noi
\textup{(i)}
Let $s_1 <-\frac{1}{4}$.
Then, 
the random linear solution $z_N (t)  
 = S(t) C_{\al, N} \P_N u_0$
defined in~\eqref{lin1}
 converges to $0$ in $C(\R_+; W^{s_1,\infty}(\T) )$ almost surely,
as $N \to \infty$.

\medskip

\noi
\textup{(ii)}
Let $s_1< -\frac{1}{4}$ and  $s_2<\frac{1}{2}$.
Then, there exist a probability space $(\wt \O,  \wt \sF, \wt \PP)$, 
$C(\R_+; W^{s_1, \infty}(\T))$-valued random variable
$\{\wt z_N\}_{N \in \N}$,  and 
$C(\R_+; W^{s_2, \infty}(\T))$-valued random variables
$\{\wt Z_N\}_{N \in \N}$ and $\wt Z$
with 
\begin{align}
\Law(\wt z_N) = \Law(z_N), \quad  
\Law(\wt Z_N) = \Law(Z_N) 
\quad\text{and}\quad 
\Law(\wt Z) = \Law(Z) 
\label{XY0}
\end{align}

\noi
such that, as $N \to \infty$, 
$(\wt z_N, \wt Z_N)$ converges to $(0, \wt Z)$
in 
$C(\R_+; W^{s_1, \infty}(\T))
\times 
C(\R_+; W^{s_2, \infty}(\T))$, $\wt \PP$-almost surely.

\end{lemma}

\begin{proof}
(i)  
We first consider the case $\al < \frac 14$.
Given $s_1 < - \frac 14$,
choose small $\dl > 0$ such that 
 $\dl < \min\big(\frac 14 - \al, - \frac 14 - s_1\big)$.
 Let $y_N = N^\dl z_N$.
Then, from \eqref{lin1} with \eqref{CN} and $|n|\le N$, we have
\begin{align}
\E\big[|\ft y_N(t, n)|^2\big]
\les \jb{n}^{- \frac 12 + 2\dl}
\label{YY1}
\end{align}

\noi
for any $n \in \Z$
and $t \in \R_+$.
By using \eqref{TT2}, we also have 
\begin{align}
\E\big[|\dl_h \ft y_N(t, n)|^2\big]
\les \jb{n}^{- \frac 12 + 2\dl}|h|^2, 
\label{YY2}
\end{align}

\noi
where $\dl_h$ is the difference operator defined in \eqref{dl1}.
Hence, it follows from Lemma~\ref{LEM:OOT} that 
 $y_N$
converges almost surely to some limiting process $y$ in $C(\R_+; W^{s_1, \infty}(\T))$
as $N \to \infty$.
Recalling that $z_N = N^{-\dl}y_N$, 
we then conclude that $z_N$
converges almost surely to $0$
in $C(\R_+; W^{s_1, \infty}(\T))$
as $N \to \infty$.

Next, we consider the case $\al = \frac 14$.
In this case, we set 
 $y_N = C_{\frac 14, N}^{-1} z_N$, 
 where $C_{\frac 14, N}$ is as in~\eqref{CN}.
Then, a simple computation shows that \eqref{YY1} and \eqref{YY2} hold
with $\dl = 0$.
Hence, by proceeding as above with Lemma \ref{LEM:OOT}, 
we conclude that 
$z_N$
converges almost surely to $0$
in $C(\R_+; W^{s_1, \infty}(\T))$
as $N \to \infty$.

\medskip

\noi
(ii)
From Part (i) of this lemma
and 
Theorem \ref{THM:1}
with together \cite[Theorem 3.9]{PB71}, 
we see that $(z_N, Z_N)$ converges in law to $(0, Z)$ 
in 
$C(\R_+; W^{s_1, \infty}(\T))
\times 
C(\R_+; W^{s_2, \infty}(\T))$
with respect to the probability
measure $\PP$ as $N \to \infty$.
Hence, the claim follows from the Skorokhod representation theorem
(Lemma~\ref{LEM:Sk}); see also Remark \ref{REM:Sk2}.
\end{proof}

We conclude this section by  presenting a proof of Theorem \ref{THM:3}.

\begin{proof}[Proof of Theorem \ref{THM:3}]

Let $\{\wt z_N\}_{N \in \N}$,  $\{\wt Z_N\}_{N \in \N}$, and $\wt Z$ be as in Lemma \ref{LEM:sto}\,(ii).
Consider the following equation:
\begin{align}
\begin{cases}
\dt \wt v_N= 
\vp(D) \wt v_N 
+ \vp(D) 
\big(
(\wt v_N +  \wt z_N +\wt Z_N )^2 - \wt z_N^2\big)\\
\wt  v_N|_{t=0} = 0, 
\end{cases}
\label{XY1}
\end{align}

\noi
which is nothing but 
\eqref{BBM4}, where $(z_N, Z_N)$ is replaced by $(\wt z_N, \wt Z_N)$.
Given $N \in \N$, 
we have $\wt z_N, \wt Z_N \in 
C(\R_+; C^\infty(\T))$ and thus
there exists a unique global-in-time solution 
$\wt v_N
\in C(\R_+; H^{\frac 34 - \eps}(\T))$ to \eqref{XY1}.
We, however, point out that there is no uniform (in $N$)
control on the $H^{\frac 34 - \eps}$-norm of $\wt v_N(t)$.

In the following, we show that the solution
$\wt v_N$ to \eqref{XY1} converges $\wt \PP$-almost surely to 
the solution $\wt v$ to 
\begin{align}
\begin{cases}
\dt   \wt v=  \vp(D) \wt v +  \vp(D) 
\big( (\wt v +  \wt Z)^2 \big) \\
 \wt v|_{t=0} =0
 \end{cases}
\label{XY3}
\end{align}

\noi
in $C(\R_+; H^{\frac 34 - \eps}(\T))$.
Note that 
\eqref{XY3} is nothing but 
\eqref{BBM4c}, where $Z$ is replaced by $\wt Z$
and, in particular, 
Proposition~\ref{PROP:2}
guarantees almost sure global well-posedness
of~\eqref{XY3}
such that $\wt  v \in C(\R_+; H^1(\T))$.
Moreover, given any finite $T \gg 1$, 
there exists an almost surely finite constant $C_\o(T) > 0$
such that  the solution $\wt v = \wt v^\o$ to \eqref{XY3} satisfies
the following growth bound:
\begin{align}
\sup_{0 \le t \le T} \| \wt v(t)\|_{H^1} \le C_\o(T).
\label{XY4a}
\end{align}

\noi
The bound \eqref{XY4a} plays a crucial role
in establishing convergence of $\wt v_N$ to $\wt v$ {\it globally} in time.

By writing \eqref{XY1} and \eqref{XY3} in the Duhamel formulation, we have 
\begin{align}
\begin{split}
\wt v_N & = \I\big(\vp(D)(\wt v_N + \wt Z_N)^2\big)+2 \I\big((\vp(D)(\wt z_N \wt Z_N)\big)
+ 2 \I\big(\vp(D)(\wt v_N \wt z_N)\big), \\
\wt v & = \I\big(\vp(D)(\wt v + \wt Z)^2\big), 
\end{split}
\label{XY2}
\end{align}

\noi
where $\I$ denotes the Duhamel integral operator defined in \eqref{exp2}.
By the definition of the compact-open topology,
it suffices to show that $\wt v_N$
converges $\wt \PP$-almost surely to $\wt v$ in $C([0, T]; H^{\frac 34 - \eps}(\T))$
for each given $T \gg1$.

Fix the target time $T \gg1$.
Let $\Si \subset \wt \O$ with $\wt \PP(\Si) = 1$
such that 
Proposition \ref{PROP:2}
(for \eqref{XY3})
and 
Lemma \ref{LEM:sto} hold.
Moreover, we assume that
for any $\o \in \Si$, 
 the equation \eqref{XY3}
is  globally well-posed
with $\wt Z = \wt Z^\o$
such that  $\wt  v \in C(\R_+; H^1(\T))$, 
satisfying 
\eqref{XY4a}.
In the following, we fix $\o \in \Si$.

Fix small $\dl > 0$.
From Lemma \ref{LEM:BOZ} followed by Lemma \ref{LEM:sto}, we have
\begin{align}
\begin{split}
\big\| \I\big((\vp(D)(\wt z_N \wt Z_N)\big)\big\|_{C_TH^{\frac 34 - \eps}_x}
& \les T 
\| \wt z_N \wt Z_N\|_{C_TH^{-\frac 14 - \eps}_x}\\
& \les T 
\| \wt z_N \|_{C_TW^{ - \frac14 - \eps, \infty}_x}
\|  \wt Z_N\|_{C_TW^{ \frac 12- \eps, \infty}_x}
\ll \dl 
\end{split}
\label{XY5}
\end{align}

\noi
for any $N \ge N_\o (T, \dl)$.
From Lemma \ref{LEM:sto} and \eqref{XY4a}, we have 
\begin{align}
\begin{split}
\big\| \I\big((\vp(D)(\wt v \wt z_N)\big)\big\|_{C_TH^{\frac 34 - \eps}_x}
& \les T 
\|  \wt v\|_{C_TH^{\frac 34 - \eps}_x}
\| \wt z_N \|_{C_TW^{- \frac14 - \eps, \infty}_x}
 \ll \dl
\end{split}
\label{XY6}
\end{align}

\noi
for any $N \ge N_\o (T, \dl)$.
Fix an interval 
$I \subset [0, T]$ with $|I| = \tau$.
Then, 
from Lemma \ref{LEM:sto}, we have 
\begin{align}
\begin{split}
\big\| \I\big((\vp(D)((\wt v_N- \wt v) \wt z_N)\big)\big\|_{C_I H^{\frac 34 - \eps}_x}
& \les \tau
\|  \wt v_N - \wt v\|_{C_I H^{\frac 34 - \eps}_x}
\| \wt z_N \|_{C_TW^{ - \frac14 - \eps, \infty}_x}\\
& \ll \tau 
\|  \wt v_N - \wt v\|_{C_IH^{\frac 34 - \eps}_x}
\end{split}
\label{XY7}
\end{align}

\noi
for any $N \ge N_\o (T, \dl)$, 
uniformly in $I \subset [0, T]$.
By H\"older's and Sobolev's inequalities, 
Lemma~\ref{LEM:sto},  and \eqref{XY4a}, we have 
\begin{align}
\begin{split}
 \big\| & \I\big(\vp(D)(\wt v_N + \wt Z_N)^2\big)
-  \I\big(\vp(D)(\wt v + \wt Z)^2\big)\big\|_{C_IH^{\frac 34 - \eps}_x}\\
& \les
\tau \Big(\| \wt v_N^2 - \wt v^2\|_{C_IL^2_x}
+ \| \wt v_N \wt Z_N  - \wt v \wt Z\|_{C_IL^2_x}
+ \| \wt Z_N^2 - \wt Z^2\|_{C_TL^2_x}\Big)\\
& \les
\tau \Big(\| \wt v_N + \wt v\|_{C_IH^{\frac 34 - \eps}_x} + 
 \| \wt Z_N\|_{C_TW^{\frac 12 - \eps, \infty}_x}\Big)
\| \wt v_N - \wt v\|_{C_I H^{\frac 34 - \eps}_x}\\
& \quad + \tau\Big(\| \wt v\|_{C_TH^1_x}
+ \| \wt Z_N  +  \wt Z\|_{C_TW^{\frac 12 - \eps, \infty}_x}\Big)
\| \wt Z_N  -  \wt Z\|_{C_TW^{\frac 12 - \eps, \infty}_x}\\
& \les
\tau \Big(C_\o(T) + 
\| \wt v_N -  \wt v\|_{C_IH^{\frac 34 - \eps}_x} \Big)
\| \wt v_N - \wt v\|_{C_I H^{\frac 34 - \eps}_x}\\
& \quad + 
\tau C_\o(T) 
\| \wt Z_N  -  \wt Z\|_{C_TW^{\frac 12 - \eps, \infty}_x}
\end{split}
\label{XY8}
\end{align}

\noi
for any $N \ge N_\o (T, \dl)$.

In the following, we assume  $N \ge N_\o (T, \dl)$
such that the estimates above hold.
Let $I_0 = [0, \tau]$
with $0 < \tau \le 1$ (to be chosen later).
Then, from \eqref{XY2}, \eqref{XY5}, 
\eqref{XY6}, \eqref{XY7}, and \eqref{XY8}, we have 
\begin{align}
\begin{split}
 \| \wt v_N - \wt v\|_{C_{I_0} H^{\frac 34 - \eps}_x}
&  \le \tau \Big(C_\o(T) + 
\| \wt v_N -  \wt v\|_{C_{I_0}H^{\frac 34 - \eps}_x} \Big)
\| \wt v_N - \wt v\|_{C_{I_0} H^{\frac 34 - \eps}_x}\\
& \quad + 
C_1(T, \o) 
\| \wt Z_N  -  \wt Z\|_{C_TW^{\frac 12 - \eps, \infty}_x}
+ c_2 \dl 
\end{split}
\label{XY9}
\end{align}

\noi
for some small $c_2>0$.
Since $v_N(0) = v(0) = 0$,
it follows from the continuity (in time)
of $\wt v_N$ and $\wt v$ that 
\begin{align}
 \| \wt v_N - \wt v\|_{C_{I_0'} H^{\frac 34 - \eps}_x} \le 1
\label{XY10}
\end{align}

\noi
on some subinterval $I_0' = I_0' (N, \o) = [0, \tau']
\subset I_0= [0, \tau]$.
Noting that \eqref{XY9} also holds on $I_0'$, 
it follows from \eqref{XY9} and \eqref{XY10}
and taking $\tau = \tau(T, \o)> 0$ sufficiently small, 
independent of $N \in \N$, 
(which also makes $\tau' > 0$ small) that 
\begin{align}
\begin{split}
 \| \wt v_N - \wt v\|_{C_{I_0'} H^{\frac 34 - \eps}_x}
& \le 2 C_1(T, \o) 
\| \wt Z_N  -  \wt Z\|_{C_TW^{\frac 12 - \eps, \infty}_x}
+ 2c_2 \dl 
\\
& \le 3 c_2 \dl \ll 1, 
\end{split}
\label{XY11}
\end{align}

\noi
where the second step follows from Lemma \ref{LEM:sto}, 
provided that  $N \ge N_\o (T, \dl)$
(by possibly taking $N_\o (T, \dl)$ larger, independently of $I_0' \subset I_0$).
Hence, by a standard continuity argument, 
we have 
\begin{align*}
 \| \wt v_N - \wt v\|_{C_{I_0} H^{\frac 34 - \eps}_x} \le 1
\end{align*}

\noi
on the  first interval $I_0$
and thus from \eqref{XY11}, we obtain
\begin{align}
 \| \wt v_N - \wt v\|_{C_{I_0} H^{\frac 34 - \eps}_x} \le 3c_2 \dl
\label{XY13}
\end{align}

\noi
for any   $N \ge N_\o (T, \dl)$.

Let $I_1 = [\tau, 2\tau]$.
 Proceeding as before, we have 
\begin{align}
\begin{split}
 \| \wt v_N - \wt v\|_{C_{I_1} H^{\frac 34 - \eps}_x}
&  \le 
\| \wt v_N(\tau) -  \wt v(\tau)\|_{H^{\frac 34 - \eps}}\\
& \quad + 
\tau \Big(C_\o(T) + 
\| \wt v_N -  \wt v\|_{C_{I_1}H^{\frac 34 - \eps}_x} \Big)
\| \wt v_N - \wt v\|_{C_{I_1} H^{\frac 34 - \eps}_x}\\
& \quad + 
C_1(T, \o) 
\| \wt Z_N  -  \wt Z\|_{C_TW^{\frac 12 - \eps, \infty}_x}
+ c_2 \dl. 
\end{split}
\label{XY14}
\end{align}

\noi
From \eqref{XY13}, we have
\begin{align*}
 \| \wt v_N (\tau) - \wt v(\tau) \|_{H^{\frac 34 - \eps}_x} \le 3c_2 \dl \ll1.
\end{align*}

\noi
Then, 
it follows from  the continuity (in time)
of $\wt v_N$ and $\wt v$ that 
\begin{align}
\| \wt v_N - \wt v\|_{C_{I_1'} H^{\frac 34 - \eps}_x} \le 1
\label{XY15}
\end{align}

\noi
on some subinterval  $I_1' = I_1'(N, \o)
= [\tau, \tau + \tau'] \subset I_1 \subset [\tau, 2\tau]$.
Then, by noting that \eqref{XY14}
also holds on $I_1'$ and proceeding with   a continuity argument 
with \eqref{XY14} and \eqref{XY15}
as before
(see also \eqref{XY11}), 
we have
\begin{align*}
 \| \wt v_N - \wt v\|_{C_{I_1} H^{\frac 34 - \eps}_x}
& \le 
2 \| \wt v_N(\tau) -  \wt v(\tau)\|_{H^{\frac 34 - \eps}}
+ 3c_2 \dl 
\\
& \le 9 c_2 \dl\ll1, 
\end{align*}

\noi
where the last step 
can be guaranteed by taking 
$N_\o(T, \dl)$ sufficiently large;
see \eqref{XY5} and~\eqref{XY6}.
We emphasize that we used the same $\tau = \tau(T, \o)$
from the first step on  $I_0$.

By iterating this process on $I_j = [j \tau,( j+1)\tau]$, 
$j = 2, \dots, \big[\frac T\tau\big]$, 
we obtain
\begin{align}
 \| \wt v_N - \wt v\|_{C_{I_j} H^{\frac 34 - \eps}_x}
& \le 
\bigg(\sum_{k = 0}^{j} 2^k\bigg)3c_2 \dl
\le 3\cdot 2^{j+1}c_2 \dl
\label{XY16}
\end{align}

\noi
for any   $N \ge N_\o (T, \dl)$, 
provided that 
we take 
$\dl > 0$ sufficiently small 
and hence we take
$N_\o (T, \dl)$ sufficiently large
such that 
\[ 3\cdot 2^{[ \frac T\tau]+1}c_2 \dl \ll 1.\]

Therefore, from \eqref{XY16}, 
we obtain
\begin{align*}
 \| \wt v_N - \wt v\|_{C_{T} H^{\frac 34 - \eps}_x}
\le 4\cdot 2^{j+1}c_2 \dl
\end{align*}

\noi
for any   $N \ge N_\o (T, \dl)$.
Since the choice of $\dl > 0$ was arbitrary, 
this proves convergence of 
$\wt v_N$ to $\wt v$
in $C([0, T]; H^{\frac 34 - \eps}(\T))$
for any $\o \in \Si$.
Since the choice of $T \gg1 $
was arbitrary, 
we thus conclude almost sure convergence of 
$\wt v_N$ to $\wt v$
in $C(\R_+; H^{\frac 34 - \eps}(\T))$.

Let  $\wt u_N = \wt z_N + \wt Z_N + \wt v_N$
and $\wt u = \wt Z + \wt v$.
Then, together 
with Lemma \ref{LEM:sto}, 
we see that $\wt u_N$ converges almost surely 
to $\wt u$  
in $C(\R_+; H^{- \frac 14 - \eps}(\T))$ as $N \to \infty$.
Note from \eqref{XY0} that $\wt u_N$
and $\wt u$ satisfy~\eqref{BBM2} and~\eqref{BBM4d} in law,\footnote{Namely, 
by replacing the random initial data $u_0$
in \eqref{BBM2} and the spatial white noise $\ze$ in \eqref{BBM4d} 
with some appropriate $\wt u_0$ and $\wt \ze$, respectively, 
such that 
$\Law (\wt u_0) = \Law (u_0)$
and 
$\Law (\wt \ze) = \Law (\ze)$.}
 respectively.
Then, from the uniqueness of solutions, 
we obtain
\begin{align*}
\Law(\wt u_N) = \Law(u_N) \qquad  
\text{and}\qquad 
\Law(\wt u) = \Law(u), 
\end{align*}

\noi
where $u_N$ and $u$ are the solutions
to 
\eqref{BBM2} and~\eqref{BBM4d}, respectively.
Therefore, we conclude that 
the solution  $u_N$ to \eqref{BBM2}
converges in law
the solution  $u$ to \eqref{BBM4d}
in $C(\R_+; H^{-\frac 14 - \eps}(\T))$ as $N \to \infty$.
This concludes the proof of Theorem \ref{THM:3}.
\end{proof}

\appendix

\section{On the stochastic BBM with a rough space-time noise}
\label{SEC:A}

In this appendix, we briefly discuss renormalization beyond variance blowup
for the stochastic BBM
\eqref{BBM6}:
\begin{align}
\dt u-\partial_{txx}u+\partial_{x}u
+\partial_{x}(u^{2})  = \jb{\dx}^{\al} \dx \xi,   
\label{XBBM0}
\end{align}

\noi
Here, $\xi$ denotes a (Gaussian) space-time white noise 
on $\R_+ \times \T$
whose space-time covariance is (formally) given by 
\begin{align*}
 \E[ \xi(t_1, x_1)\xi(t_2, x_2) ] = \dl (t_1 - t_2) \dl(x_1 - x_2) 
\end{align*} 

\noi
for $t_1, t_2 \in \R_+$  and $x_1, x_2 \in \T$
with $\dl$ denoting the Dirac delta function.
As mentioned in Remark \ref{REM:SBBM1}\,(ii), 
 a slight modification of the argument in \cite{Forlano}
 shows 
that  variance blowup occurs
for  $\al \ge \frac 34$.
Our goal in this appendix is to prove an analogue
of Theorem~\ref{THM:3}
for~\eqref{XBBM0}.

\subsection{Renormalization of the stochastic BBM 
beyond variance blowup}
\label{SUBSEC:A1}

As in Hairer's work \cite{Hairer} (see \eqref{KPZ2}), we
introduce a vanishing multiplicative renormalization constant on the noise:
\begin{align}
\begin{cases}
\dt u_N-\partial_{txx}u_N+\partial_{x}u_N
+\partial_{x}(u_N^{2})  = C_{1-\al, N}  \P_N \jb{\dx}^\al  \dx \xi\\
u|_{t = 0} = v_0, 
\end{cases}
\label{XBBM1}
\end{align}

\noi
where $C_{1-\al, N}$ is as in \eqref{CN} (with $\al$ replaced by $1-\al$).
With $\vp(D)$ as in~\eqref{phi1}, we can rewrite \eqref{XBBM1}
as 
\begin{align}
\begin{cases}
\dt u_N= \vp(D) u_N
+\vp(D)(u_N^{2})  -  C_{1-\al, N}  \P_N \vp(D) \jb{\dx}^\al  \xi\\
u|_{t = 0} = v_0.
\end{cases}
\label{XBBM2}
\end{align}

Let $\zs_N$ be the solution to the following renormalized linear stochastic equation:
\begin{align*}
\begin{cases}
\dt \zs_N= \vp(D) \zs_N
 -  C_{1-\al, N}  \P_N \vp(D) \jb{\dx}^\al  \xi\\
\zs_N|_{t = 0} = 0.
\end{cases}
\end{align*}

\noi
Namely, $\zs_N$ is the renormalized stochastic convolution given by
\begin{align}
\zs_N(t) = -  C_{1-\al, N}  \int_0^t 
S(t - t')
 \P_N \vp(D) \jb{\dx}^\al  \xi (dt'), 
\label{sto2}
\end{align}

\noi
where $S(t) = e^{t\vp(D)}$.
See Section \ref{SUBSEC:A2}
for a precise meaning of \eqref{sto2}.

We then define $\Ys_N$ and $\Zs_N$ by 
\begin{align}
\Ys_N =  - \P_{\ne 0}(\zs_N^2)
\qquad \text{and}\qquad
\Zs_N 
=  \I\big(\vp(D) (\zs_N)^2\big)
= - \I(\vp(D) \Ys_N), 
\label{sto2a}
\end{align}

\noi
where 
$ \P_{\ne 0}$ is the (spatial) frequency projection onto non-zero frequencies
and 
$\I$ denotes the Duhamel integral operator defined in \eqref{exp2}.
We note that, in order to obtain the limiting equation \eqref{XBBM4}, 
we need to study convergence in law of 
$\Ys_N$, not of $\Zs_N$.

\begin{proposition}
\label{PROP:4}

Let $\al\ge \frac 34$.
Given any $s <-  \frac 12$, 
$\Ys_N = -  \P_{\ne 0}(\zs_N^2)$ defined in~\eqref{sto2a}
converges in law to a unique limit  $\upze$ 
in $C(\R_+; W^{s,\infty}(\T))$
as $N \to \infty$, 
where $\upze$ is a centered 
Gaussian process
whose space-time covariance  \textup{(}when tested against test functions
$\psi_1, \psi_2 \in \D(\T)$\textup{)} is given by 
\begin{align}
\begin{aligned}
&   \E \Big[\jb{\upze (t_1), \psi_1} \jb{ \upze (t_2), \psi_2}\Big]
 =  
(t_1 \wedge t_2)^2
 \sum_{n \in \Z^*}
\cj{\ft \psi_1(n)} \ft \psi_2(n)
\end{aligned}
\label{sto2c}
\end{align}

\noi
for any $t_1, t_2 \in \R_+$
and $\psi_1, \psi_2 \in \D(\T)$.

\end{proposition}

We present a proof of Proposition \ref{PROP:4} 
in the next subsection.
We note that,   from \eqref{sto2c}, 
we can write 
the limiting process $\upze$ as 
\begin{align*}
\upze (t)  = \sum_{n\in \Z^*}
\int_0^t \sqrt {2t'} d B_n(t') e_n
\end{align*}

\noi
for each $t \in \R_+$,  
where 
 $\{ B_n \}_{n \in \Z}$ is a family of mutually independent complex-valued
Brownian motions with $\text{Var}(B_n(t)) = t$ conditioned  that $B_{-n} = \cj{B_n}$, $n \in \Z$. 
%
%
%
In particular, for each $t \in \R_+$, 
 $\upze(t)$ is a scalar multiple of a spatial
white noise 
 with variance  $t^2$
(with spatial mean~$0$).

As a corollary to Proposition \ref{PROP:4}, we have the following convergence result
for $\Zs_N$.

\begin{corollary}\label{COR:Z}

Let $\al\ge \frac 34$.
Given any $s <  \frac 12$, 
$\Zs_N$ defined in~\eqref{sto2a}
converges in law to a unique limit  $\Zs$ 
in $C(\R_+; W^{s,\infty}(\T))$
as $N \to \infty$.
Here,  $\Zs$ is the unique solution to the following equation\textup{:}
\begin{align}
\begin{cases}
\dt \Zs-\partial_{txx}\Zs+\dx \Zs
 = \dx \upze  \\
\Zs|_{t=0} =0, 
\end{cases}
\label{XBBM4a}
\end{align}

\noi
where 
 $\upze$ is as in Proposition \ref{PROP:4}.

\end{corollary}

\medskip

As in \eqref{exp3x}, 
consider the second order expansion
for the solution $u_N$ to \eqref{XBBM2}:
\begin{align*}
u_N
&  =  \zs_N +\Zs_N + v_N, 
\end{align*}

\noi
where $\zs_N$ and $\Zs_N$ are as in \eqref{sto2} and \eqref{sto2a}.
Then, the remainder term 
$v_N = u_N - \zs_N -\Zs_N$ satisfies the following equation: 
\begin{align}
\begin{cases}
\dt v_N= 
\vp(D) v_N 
+ \vp(D) 
\big(
(v_N +  \zs_N +\Zs_N )^2 - \zs_N^2\big)\\
v_N|_{t=0} = v_0.
\end{cases}
\label{XBBM4b}
\end{align}

\noi
As in the random initial data case, 
$\zs_N$ converges to $0$ as $N \to \infty$; see Lemma \ref{LEM:sto1}.
Then, together with Corollary \ref{COR:Z}, we
formally obtain the following limiting equation:
\begin{align}
\begin{cases}
\dt   v=  \vp(D) v +  \vp(D) 
\big( (v +  \Zs)^2 \big) \\
 v|_{t=0} =v_0.
 \end{cases}
\label{XBBM4c}
\end{align}

\noi
Then,   by setting 
$u = \Zs + v$, 
we obtain, from \eqref{XBBM4a} and \eqref{XBBM4b}, 
the following limiting equation:
\begin{align}
\begin{cases}
\dt u-\partial_{txx}u+\partial_{x}u
+\partial_{x}(u^{2})  = \dx \upze  \\
u|_{t = 0} = v_0.
\end{cases}
\label{XBBM4}
\end{align}

Let $v_0 \in H^1(\T)$.
Then, by the same proof, we see that an analogue of  Proposition \ref{PROP:2}
holds for  the limiting equation \eqref{XBBM4}.
In particular, \eqref{XBBM4} is almost surely globally well-posed
with the solution $v$ to \eqref{XBBM4c}
in the class
$C(\R_+; H^1(\T))$.

Finally, we state the convergence result.

\begin{theorem}\label{THM:4a}

Let $\al\ge \frac34$
and $v_0 \in H^1(\T)$.
As $N \to \infty$, 
the solution $u_N$
to \eqref{XBBM1}
with the renormalized noise
converges in law
to the solution $u$ to \eqref{XBBM4}
in $C(\R_+; H^{- \frac 14 - \eps}(\T))$.

\end{theorem}

Theorem \ref{THM:4a} follows
from 
a straightforward modification of the proof of Theorem~\ref{THM:3}
presented in Section \ref{SEC:conv}
with Corollary \ref{COR:Z}
and the convergence property of the renormalized stochastic convolution
$\zs_N$ in \eqref{sto2} (see Lemma \ref{LEM:sto1} below), 
and thus we omit details.
See also 
Remark \ref{REM:IV}.

\medskip

Lastly, we state an analogue of Theorem \ref{THM:3a}
in the current context.
More precisely, 
for $\al \ge \frac 34$, consider the following 
 weakly interacting stochastic  BBM:
\begin{align}
\begin{cases}
\dt \uu_N-\partial_{txx}\uu_N+\partial_{x}\uu_N
+C_{1- \al, N}^2 \partial_{x}(\uu_N^{2})  = 
\P_N \jb{\dx}^\al  \dx \xi\\
\uu_N|_{t=0} = v_0.
\end{cases}
\label{XBBM8}
\end{align}

\noi
Then, by considering the first order expansion 
and formally taking a limit as $N \to \infty$, 
we obtain the following limiting equation:
\begin{align}
\begin{cases}
\dt \uu-\partial_{txx}\uu+ \partial_{x}\uu
  = \jb{\dx}^\al  \dx  \xi + \dx \upze  \\
\uu|_{t = 0} = v_0, 
\end{cases}
\label{XBBM9}
\end{align}

\noi
where 
$\upze$ is as in Proposition \ref{PROP:4}, 
assumed to be 
 independent of the space-time white noise~$\xi$.

\begin{theorem}\label{THM:6}

Let $\al\ge \frac34$
and $v_0 \in H^{\frac 12 - \eps} (\T)$.
As $N \to \infty$, 
the solution $\uu_N$
to~\eqref{XBBM8}
with the renormalized nonlinearity
converges in law
to the solution $\uu$ to \eqref{XBBM9}
in $C(\R_+; H^{\frac 12 -  \al  - \eps}(\T))$.

\end{theorem}


Define  $\Psi$ by 
\begin{align*}
\Psi(t) = -   \int_0^t 
S(t - t')
 \vp(D) \jb{\dx}^\al  \xi (dt').
\end{align*}

\noi
It is easy to see that $\Psi $ almost surely belongs
to  $C(\R_+; W^{\frac 12 - \al  - \eps, \infty}(\T))$.
Then, by writing
 the solution $\uu$ to \eqref{XBBM9}
in the first order expansion: $\uu = \Psi + \vv$, 
we see that $\vv$ satisfies
\begin{align*}
\vv(t) = S(t) v_0 + \Zs, 
\end{align*}

\noi
where $\Zs$ is as in Corollary \ref{COR:Z}, 
almost surely belonging
to $C(\R_+; W^{\frac 12- \eps,\infty}(\T))$.
This gives the regularity $\frac 12 - \eps$
for the initial data $v_0$ in Theorem \ref{THM:6}.

With this observation, 
in view of Corollary \ref{COR:Z}, 
Theorem \ref{THM:6}
follows from 
a straightforward modification of the proof of Theorems~\ref{THM:3}, 
\ref{THM:3a}, and \ref{THM:4a} 
(see also Remark \ref{REM:indep})
and thus we omit details.

\subsection{Convergence properties of the stochastic terms}
\label{SUBSEC:A2}

In this subsection, we establish convergence properties
of the stochastic terms $\zs_N$ in \eqref{sto2} and $\Ys_N$
in \eqref{sto2a}.

Let us first give a precise meaning of the stochastic convolution $\zs_N$ in~\eqref{sto2}.
 Let~$W$ denote a cylindrical Wiener process on $L^2(\T)$:
\begin{align}
\mathcal W(t)
 = \sum_{n \in \Z} B_n (t) e_n, 
\label{BM1}
\end{align}

\noi
where 
$\{ B_n \}_{n \in \Z}$ 
is defined by 
$B_n(t) = \jb{\xi, \ind_{[0, t]} \cdot e_n}_{t, x}$.
Here, $\jb{\,\cdot, \cdot\,}_{t, x}$ denotes 
the duality pairing on $ \R_+\times \T$.
As a result, 
we see that $\{ B_n \}_{n \in \Z}$ is a family of mutually independent complex-valued
Brownian motions conditioned  that $B_{-n} = \cj{B_n}$, $n \in \Z$. 
Note that we have
 \[\text{Var}(B_n(t)) = \E\big[
 \jb{\xi, \ind_{[0, t]} \cdot e_n}_{t, x}\cj{\jb{\xi, \ind_{[0, t]} \cdot e_n}_{ t, x}}
 \big] = \|\ind_{[0, t]} \cdot e_n\|_{L^2_{t, x}}^2 = t\]

\noi
 for any $n \in \Z$.
Then, we can write the stochastic convolution $\zs_N$ in \eqref{sto2}
as 
\begin{align}
\begin{split}
\zs_N(t) 
& = -  C_{1-\al, N}  \int_0^t 
S(t - t')
 \P_N \vp(D) \jb{\dx}^\al  d\mathcal W(t')\\
& = 
-  C_{1-\al, N}  
\sum_{ |n|\le N}
e^{t\vp(n)} \If_n(t) e_n, 
\end{split}
\label{sto3}
\end{align}

\noi
where $\If_n(t)$, $ n \in \Z$,  is the Wiener integral given by 
\begin{align}
\If_n(t) =  \int_0^t e^{- t'\vp(n)} \vp(n)\jb{n}^\al dB_n(t').
\label{sto4}
\end{align}

\noi
Recalling that $\vp(n)|_{ n = 0} = 0$, 
we have $\If_0 \equiv 0$.
Moreover, we have $\If_{-n} = \cj{\If_n}$
for any $n \in \Z^*$.

\begin{lemma}\label{LEM:sto1}
Let $\al \ge \frac 34$.
Then, 
given any  $s <-\frac{1}{4}$, 
the renormalized stochastic convolution $\zs_N$ defined in~\eqref{sto2}
 converges to $0$ in $C(\R_+; W^{s,\infty}(\T) )$ almost surely,
as $N \to \infty$.

\end{lemma}

\begin{proof}

Note that $\If_n(t)$ is a mean-zero complex-valued Gaussian random variable
with variance $t |\vp(n)|^2 \jb{n}^{2\al}$.
Then, the claim follows from arguing as in the proof of
Lemma \ref{LEM:sto}\,(i)
with 
Lemma \ref{LEM:OOT}.
We omit details.
\end{proof}

We now turn to the convergence property of $\Ys_N$ in \eqref{sto2a}.
As in Section \ref{SEC:Z}, 
Proposition~\ref{PROP:4}
follows once we establish the following two lemmas
together with 
 \cite[Definition I.3.5]{RY} which guarantees that
$\upze$ is a centered Gaussian process.

\begin{lemma}[uniqueness]\label{LEM:5}
Let $\al\ge \frac 34$.
Let $\Ys_N$ and $\upze$ be as in \eqref{sto2a} and 
Proposition~\ref{PROP:4}, respectively.
Given $m \in \N$, let 
 $0\le t_1\le  \cdots \le t_m$.
Then, 
$\big\{\Ys_N(t_j)\big\}_{j = 1}^m$
converges in law
to a centered Gaussian vector
$\big\{\upze(t_j)\big\}_{j = 1}^m$
in $(\D'(\T))^{\otimes m}$
as $N \to \infty$.

\end{lemma}

\begin{lemma}[tightness]\label{LEM:tight2}
Let $\al\ge \frac 34$
and $\Ys_N$ be as in \eqref{sto2a}.
Then, given any  $s < - \frac 12$, 
the sequence $\{\Ys_N\}_{N \in \N}$ is tight in 
 $C(\R_+;W^{s,\infty}(\T))$.

\end{lemma}

Lemma \ref{LEM:tight2} on tightness
follows from 
a slight modification of  the proof of Proposition \ref{PROP:tight}, 
using \eqref{sto2a}, \eqref{sto3}, and \eqref{sto4}.
Since the required modification
is straightforward, we omit details.
In the following, we focus on Lemma \ref{LEM:5}
which is 
an analogue
of Proposition \ref{PROP:uniq}.

\begin{proof} [Proof of Lemma \ref{LEM:5}]

Lemma \ref{LEM:5} follows
from a slight modification of the proof of Proposition \ref{PROP:uniq}, 
based on Lemmas \ref{LEM:COV}
and \ref{LEM:uniq2}.
Thus, we keep our discussion brief here.

By proceeding as in the proof of 
Lemma \ref{LEM:COV}
with \eqref{sto2a}, \eqref{sto3}, and \eqref{sto4}, 
it is easy to see that 
\begin{align*}
\E \big[\jb{\Ys_N (t), \psi} \big] = 0
\end{align*}

\noi
for any $N \in \N$, $t \in \R_+$,  and $\psi \in \D(\T)$.
By taking $N \to \infty$, 
we see that  a limit is a centered process (if it exists).

\medskip

\noi
{\bf $\bul$ Step 1:}
Let $t_1, t_2 \in \R_+$ and $\psi_1, \psi_2 \in \D(\T)$.
Proceeding as in \eqref{con4}
with \eqref{sto2a} and \eqref{sto3}, we have 
\begin{align}
\begin{aligned}
\E& \Big[\jb{\Ys_N (t_1), \psi_1} \jb{ \Ys_N (t_2), \psi_2}\Big]\\
&  =  C_{1-\al, N}^4 
\sum_{n, m\in \Z^*}
\cj{\ft \psi_1(n)} \ft \psi_2(m)\\
& \quad \times 
\sum_{\substack{n= n_1+n_2\\0 < |n_1|, |n_2|\leq N}} 
\sum_{\substack{m= m_1+m_2\\0 < |m_1|, |m_2|\leq N}} 
e^{t_1(\vp(n_1)+ \vp(n_2))} 
e^{-t_2(\vp(m_1)+ \vp(m_2))} \\
& \hphantom{XXXXXXXXXXXXX} \times 
\E\Big[\If_{n_1}(t_1) \If_{n_2}(t_1) \cj{\If_{m_1}(t_2) \If_{m_2}(t_2)}\Big].
\end{aligned}
\label{Ycon4}
\end{align}

\noi
A straightforward modification of \eqref{con5}
for the Wiener integral $\If_{n}$ in 
\eqref{sto4} yields
\begin{align}
\sum_{\substack{n= m_1+m_2\\0 < |m_1|, |m_2|\leq N}} 
\E\Big[\If_{n_1}(t_1) \If_{n_2}(t_1) \cj{\If_{m_1}(t_2) \If_{m_2}(t_2)}\Big]
= 2 (t_1 \wedge t_2)^2
\prod_{j = 1}^2 |\vp(n_j)|^2 \jb{n_j}^{2\al}
\label{Ycon5}
\end{align}

\noi
for any  $n_1, n_2 \in \Z^*$ with $|n_1|, |n_2|\le N$
and $n = n_1 + n_2 \ne 0$.
Thus, from \eqref{Ycon4} and \eqref{Ycon5}, 
we have 
\begin{align}
\begin{aligned}
& \lim_{N \to \infty} \E \Big[\jb{\Ys_N (t_1), \psi_1} \jb{ \Ys_N (t_2), \psi_2}\Big]\\
& \quad 
 =  2
(t_1 \wedge t_2)^2
 \sum_{n \in \Z^*}
\cj{\ft \psi_1(n)} \ft \psi_2(n) \\
& \quad 
 \quad \times 
 \lim_{N \to \infty} C_{1- \al, N}^4
\sum_{\substack{n= n_1+n_2\\|n_1|, |n_2|\leq N}} 
e^{(t_1- t_2)(\vp(n_1)+ \vp(n_2))} 
\prod_{j = 1}^2 |\vp(n_j)|^2 \jb{n_j}^{2\al}.
\end{aligned}
\label{Ycon6}
\end{align}

\noi
Arguing as in 
\eqref{con7}, \eqref{con8}, and \eqref{con9} with \eqref{phi2}, we have 
\begin{align}
\begin{split}
& \sum_{\substack{n= n_1+n_2\\|n_1|, |n_2|\leq N}} 
e^{(t_1- t_2)(\vp(n_1)+ \vp(n_2))} 
\prod_{j = 1}^2 |\vp(n_j)|^2 \jb{n_j}^{2\al}\\
& \quad = 
 \sum_{\substack{n= n_1+n_2\\|n_1|, |n_2|\leq N}} 
\prod_{j = 1}^2 |\vp(n_j)|^2 \jb{n_j}^{2\al}
+ O(e^{\max(t_1, t_2)}).
\end{split}
\label{Ycon6a}
\end{align}

\noi
Moreover, by proceeding as in Remark \ref{REM:CA1}
(see also the last part of the proof of Lemma \ref{LEM:COV}) with \eqref{phi2}, we have 
\begin{align}
\begin{split}
&  \lim_{N \to \infty}2 C_{1- \al, N}^4 \sum_{\substack{n= n_1+n_2\\|n_1|, |n_2|\leq N}} 
\prod_{j = 1}^2 |\vp(n_j)|^2 \jb{n_j}^{2\al}\\
& = 
 \lim_{N \to \infty}2 C_{1- \al, N}^4 \sum_{\substack{n= n_1+n_2\\
 N^{1-\eps}\ll |n_1|, |n_2|\leq N}} 
\prod_{j = 1}^2 |\vp(n_j)|^2 \jb{n_j}^{2\al}\\
& = \lim_{N\to \infty} 
C_{1- \al, N}^4 \sum_{\substack{n=n_1+n_2 \\  |n_1|,|n_2|\le N }} 
\frac{2}{\jb{n_1}^{2- 2\al} \jb{n_2}^{2- 2\al}} \\
&= 1, 
\end{split}
\label{Ycon6b}
\end{align}

\noi
where the last step follows from \eqref{con11} and \eqref{AA5}.
Hence, from \eqref{Ycon6}, 
 \eqref{Ycon6a}, and  \eqref{Ycon6b}
 with \eqref{sto2c}, 
we have 
\begin{align*}
 \lim_{N \to \infty} \E \Big[\jb{\Ys_N (t_1), \psi_1} \jb{ \Ys_N (t_2), \psi_2}\Big]
&  =  
(t_1 \wedge t_2)^2
 \sum_{n \in \Z^*}
\cj{\ft \psi_1(n)} \ft \psi_2(n) \\
& =   \E \Big[\jb{\upze (t_1), \psi_1} \jb{ \upze (t_2), \psi_2}\Big].
\end{align*}

Given $0 \le t_1 \le \dots \le t_m$ and $\{ \psi_j \}_{j = 1}^m \subset \D(\T)$, 
define 
an $m\times m$  matrix  $\Cs=(\Cs_{ij})_{1\le i,j\le m}$ 
by setting 
\begin{align*}
\Cs_{ij} =  \lim_{N \to \infty} \E \Big[\jb{\Ys_N (t_i ), \psi_i} \jb{ \Ys_N (t_j), \psi_j}\Big].
\end{align*}

\noi
Then,  from the discussion above
and the fact that the set of positive semi-definite matrices is closed, 
we see that the matrix $\Cs$ is 
positive semi-definite.

\medskip

\noi
{\bf $\bul$ Step 2:}
In order to establish an analogue of Lemma \ref{LEM:uniq2}, 
we need to modify the setup 
for  
the fourth moment theorem (Lemma \ref{LEM:FMT})
from $L^2(\T)$ to $L^2(\R_+ \times \T)$.
Define a centered Gaussian process
$W=\{W(f): f\in L^2(\R_+\times \T)\}$
by setting
\begin{align}
W(f) = \sum_{n\in \Z} \int_{\R_+} \ft f (t_n, n) dB_n(t_n), 
\label{BM2}
\end{align}

\noi
where $B_n$ is as in \eqref{BM1}.
Then, from  the independence of $\{B_n \}_{n \in \Z}$
(modulo the constraint $B_{-n} = \cj{B_n}$)
and the basic property of  Wiener integrals, we have 
\begin{align*}
\E[W(f_1)W(f_2)] = \jb{f_1,f_2}_{L^2(\R_+\times \T)}
\end{align*}

\noi
for any $f_1, f_2 \in L^2(\R_+\times \T)$.
Namely, $W$ is isonormal.
With this isonormal Gaussian process $W$, 
we can define 
the multiple stochastic integral
$I_k$ by \eqref{WW2}
for 
$f\in L^2(\R_+\times \T)$ 
with $\|f\|_{L^2(\R_+ \times \T)}=1$,
which 
can be extended to a linear isometry 
between $L^2_\text{sym}((\R_+\times \T)^k)$ 
and $\H_k$, 
where $L^2_\text{sym}((\R_+\times \T)^k)$
is the closed subspace of 
$L^2((\R_+\times \T)^k)$ 
consisting of symmetric functions, 
equipped with the scaled norm:
$\|f\|_{L^2_\text{sym}((\R_+\times \T)^k)}
= \sqrt{k!} \|f\|_{L^2((\R_+\times \T)^k)}$.
For general 
$f\in L^2((\R_+\times \T)^k)$, 
we define the multiple stochastic integral $I_k(f)$
by \eqref{WW3}, 
where we extend the symmetrization~\eqref{WW4}
to the current space-time setting.
We also extend the contraction~\eqref{N3a} to the current space-time setting
in an obvious manner.
See \cite[Appendix~B]{OWZ}
for a further discussion on multiple stochastic integrals
and the references therein.

Fix $ t \in \R_+$ and $\psi \in \D(\T)$.
Then, from \eqref{sto2a}, \eqref{sto3}, 
and \eqref{sto4}
with 
 \eqref{WW2} and \eqref{BM2}, 
we write 
 $\jb{\Ys_N(t), \psi}$ as
\begin{align*}
\jb{\Ys_N(t), \psi}
= I_2(\fns^t),
\end{align*}

\noi
where $\fns^t = \fns^t(t_1, t_2, x_1, x_2) \in L^2((\R_+ \times \T)^2)$ is a symmetric function
whose Fourier transform (in $x_1, x_2$) is given by 
\begin{align*}
\ft {\fns^t}(t_1, t_2, n_1, n_2)
&
 = - 
 \ind_{0 <  |n_1|, |n_2|\leq N}
 \cdot 
 C_{1-\al, N}^2
 e^{t \vp(n_1 + n_2)} \cj{\ft \psi(n_1+n_2)}\\
& \quad 
\times \prod_{j = 1}^2 \ind_{[0, t]}(t_j)
e^{-t_j \vp(n_j)}
 |\vp(n_j)| \jb{n_j}^{\al}
\end{align*}

\noi
for $t_1, t_2 \in \R_+$ and $n_1, n_2\in \Z$.
By noting 
that $ |\vp(n_j)| \jb{n_j}^{\al}\les \jb{n_j}^{\al-1}$, 
a straightforward modification 
of \eqref{B2b} (where we replace $\al$ by $1- \al$)
yields
\begin{align*}
\lim_{N\to\infty} \|  \fns^t \otimes_1 \fns^t \|_{L^2((\R_+\times \T)^2) } =0.
\end{align*}

\noi
Hence, by  proceeding as in the proof of Lemma~\ref{LEM:uniq2}
with Step 1, the fourth moment theorem (Lemma~\ref{LEM:FMT}
but adapted to the current space-time setting), 
and Lemma \ref{LEM:BDW}, 
we obtain the claimed convergence.
This concludes the proof of Lemma  \ref{LEM:5}.
\end{proof}

\begin{ackno}\rm

The authors would like to thank Martin Hairer for a valuable conversation
at the P4 conference in 2019.
They are also grateful to 
Guangqu Zheng for a discussion on the fourth moment method.
T.O.~was supported by the European Research Council (grant no.~864138 ``SingStochDispDyn'')
and by the EPSRC Mathematical Sciences Small Grant
(grant no. EP/Y033507/1).
N.T. was partially supported by the ANR
project Smooth ANR-22-CE40-0017.

\end{ackno}

\end{document}